\documentclass[11pt]{amsart}
\usepackage{graphicx}              % to include figures
\usepackage{amsmath}              % great math stuff
\usepackage{amsfonts}              % for blackboard bold, etc
\usepackage{amsthm}                % better theorem environments
\usepackage{color}
\usepackage{tabulary}
\usepackage{caption}
\usepackage{subcaption}
\usepackage{amssymb}
\usepackage{todonotes}
\theoremstyle{plain}
\usepackage{mathtools}
\theoremstyle{plain}
\newtheorem{thm}{Theorem}
\newtheorem{theorem}{Theorem}[section]

\newtheorem{proposition}[theorem]{Proposition}
\newtheorem{lemma}[theorem]{Lemma}
\newtheorem{claim}[theorem]{Claim}

\newtheorem{corollary}[theorem]{Corollary}
\theoremstyle{definition}
\newtheorem{definition}[theorem]{Definition}
\newtheorem{remark}[theorem]{Remark}

\newcommand{\nc}{\newcommand}
\nc{\dmo}{\DeclareMathOperator}

\nc{\Q}{\mathbb{Q}}
\nc{\R}{\mathbb{R}}
\nc{\Z}{\mathbb{Z}}
\nc{\N}{\mathbb{N}}
\nc{\C}{\mathbb{C}}
\nc{\cS}{\mathcal{S}}
\nc{\iso}{\cong}
\dmo{\Mod}{Mod}
\dmo{\Diff}{Diff}
\dmo{\Homeo}{Homeo}
\dmo{\dist}{dist}
\dmo\BDiff{BDiff}
\dmo\SO{SO}
\dmo\slide{sl}
\dmo\im{im}
\dmo\id{id}
\dmo\Fix{Fix}
\dmo\Out{Out}
\dmo{\T}{\mathcal{T}}
\dmo{\Te}{\mathcal{T}^{\epsilon}}
\dmo{\M}{\mathcal{M}}
\dmo{\Me}{\mathcal{M}^{\epsilon}}

\renewcommand{\epsilon}{\varepsilon}
\nc{\coloneq}{\mathrel{\mathop:}\mkern-1.2mu=}
\nc{\margin}[1]{\marginpar{\scriptsize #1}}
\nc{\para}[1]{\bigskip\noindent\textbf{#1}}

\begin{document}

\title[Smallest eigenvalue of hyperbolic 3-manifolds]{The smallest positive eigenvalue of fibered hyperbolic 3-manifolds}
\author{Hyungryul Baik, Ilya Gekhtman and Ursula Hamenst\"{a}dt}

\thanks
{
Research of all three authors 
supported by ERC grant 10160104}
%\\
%e-mail: ursula@math.uni-bonn.de}
\date{August 26, 2016}

\begin{abstract}
We study the smallest positive eigenvalue $\lambda_1(M)$ of the
Laplace-Beltrami operator on a closed hyperbolic 3-manifold $M$ which 
fibers over the circle, with fiber a closed surface of genus $g\geq 2$.
 We show the existence of a constant $C>0$ only depending on $g$ 
so that 
$\lambda_1(M)\in [C^{-1}/{\rm vol}(M)^2,
C\log {\rm vol}(M)/{\rm vol}(M)^{2^{2g-2}/(2^{2g-2}-1)}]$ and that this 
estimate is essentially sharp. We show that if 
$M$ is typical or random, then we have $\lambda_1(M)\in [C^{-1}/{\rm vol}(M)^2,C/{\rm vol}(M)^2]$. This rests on a result of independent interest about reccurence properties of axes of random pseudo-Anosov elements. 

\end{abstract}

\maketitle

\tableofcontents

\section{Introduction}

The smallest positive eigenvalue $\lambda_1(M)$ 
of the Laplace-Beltrami operator on a closed Riemannian manifold $M$
equals the infimum of the Rayleigh quotients  
$$ \lambda_1(M) = \inf\limits_{f \in C_m^\infty(M)} \dfrac{\int_M || \nabla f ||^2 dM }{ \int_M f^2 dM },$$
where $C_m^\infty(M)$ denotes 
the vector space of smooth functions $f$ on $M$ with $\int_M fdM=0$.

For closed hyperbolic surfaces $S$ of fixed genus $g\geq 2$
and hence of fixed volume, this 
eigenvalue can be arbitrarily close to zero if there is a
separating short geodesic on $S$. In fact,
$\lambda_{2g-3}(S)$ can be arbitrarily small
(Theorem 8.1.3 of \cite{B92}). 
But for a
closed hyperbolic 3-manifold $M$, 
Schoen \cite{S82} established the existence of a universal and explicit 
constant $b_1>0$ such that
\begin{equation}\label{lower}\lambda_1(M)\geq 
\frac{b_1}{{\rm vol}(M)^2}.\end{equation}
The same lower
bound holds true for hyperbolic 3-manifolds of finite volume
\cite{DR86}.

On the other hand, Buser \cite{Bu82} showed
that the so-called \emph{Cheeger constant} $h(M)$ of $M$
can be used to give an upper estimate for
$\lambda_1(M)$ by
\[\lambda_1(M)\leq b_2(h(M)+h^2(M))\]
where $b_2>0$ is a universal constant
(which in a more general setting depends on the
dimension and a lower bound on the Ricci curvature).

Lackenby \cite{L06} related the Cheeger constant $h(M)$ 
to the \emph{Heegaard Euler characteristic} $\chi_H(M)$ of $M$.
He showed that
\[h(M)\leq \frac{4\pi \vert \chi_H(M)\vert }{{\rm vol}(M)}.\]
If we denote by ${\rm genus}(M)$ the more familiar
Heegaard genus of $M$, then we have
$\chi_H(M)=2-2{\rm genus}(M)$.

Since there is a positive lower bound for the
volume of a hyperbolic 3-manifold,  
these results can be summarized 
as follows. For every $g>0$ there exists a constant
$b_3(g)>0$ with the following property. Let $M$ be a closed
hyperbolic 3-manifold of Heegaard genus at most $g$; then
\[\frac{b_1}{{\rm vol}(M)^2}\leq 
\lambda_1(M)\leq \frac{b_3(g)}{{\rm vol}(M)}.\]

For manifolds $M$ with a given lower bound of the injectivity
radius, there is more precise information.
Namely, White \cite{White} proved that there exists a number $b_4=b_4(g,\epsilon)>0$
such that
\[\lambda_1(M)\leq \frac{b_4(g,\epsilon)}{{\rm vol}(M)^2}\]
for all closed hyperbolic 3-manifolds of Heegaard genus at most
$g$ and injectivity radius at least $\epsilon$.
The existence of expander families yield that the
dependence of $b_4(g,\epsilon)$ on $g$ is necessary.
We refer to \cite{H16} for a more complete discussion.

In this work we are interested in $\lambda_1(M)$ for a 
closed hyperbolic three-manifold $M$ 
which fibers over the circle, with fiber a 
closed surface $S$ of genus $g\geq 2$.
Such a manifold  
can be described as a mapping torus of 
a pseudo-Anosov diffeomorphism of $S$, in particular,
there are infinitely many such mapping tori. 
The Heegaard genus of a mapping torus of genus $g$ 
is not bigger than $2g+1$ (see \cite{L06} for references).

Our first goal is to give an essentially 
sharp upper bound for $\lambda_1(M)$
for hyperbolic mapping tori $M$ of fibre genus $g$. We prove.

\begin{thm}\label{thm1}
For every $g\geq 2$ there exists a constant $C_1=C_1(g)>0$ with the following 
property. 
\begin{enumerate}
\item 
Let $M$ be a hyperbolic mapping torus of genus $g$; then
\[\lambda_1(M)\leq \frac{C_1\log {\rm vol}(M)}{{\rm vol}(M)^{2^{2g-2}/(2^{2g-2}-1)}}.\]
\item There exists a sequence $M_i$ of hyperbolic mapping tori of genus $g$
with ${\rm vol}(M_i)\to \infty$ 
such that 
\[\lambda_1(M_i)\geq \frac{C_1^{-1}}
{{\rm vol}(M_i)^{2^{2g-2}/(2^{2g-2}-1)}}.\]
\end{enumerate}
\end{thm}

We believe that Theorem \ref{thm1} easily generalizes to hyperbolic 
mapping tori of non-exeptional surfaces of 
finite type with punctures, but we did not check the details.
By the work of White 
\cite{White}, the injectivity radius of the 
examples in the second part of the
above theorem tends to zero with $i$. 

The estimates in part (1) and (2) of the 
theorem differ by a factor $\log {\rm vol}(M)$. 
This deviation arises as follows.
Any closed hyperbolic 3-manifold $M$ admits a
\emph{thick-thin decomposition}
$M=M_{\rm thick}\cup M_{\rm thin}$ where for some
small but fixed number $\epsilon >0$, 
$M_{\rm thin}$ consists of all points of 
injectivity radius smaller than $\epsilon$, 
and $M_{\rm thick}=M-M_{\rm thin}$. 

We estimate effectively 
the smallest eigenvalue
$\lambda_1(M_{\rm thick})$ of $M_{\rm thick}$ 
with Neumann boundary conditions as a function of the
volume.
We then use a result of \cite{H16}: 
There exists  
a universal constant $b>0$ such that
\[b^{-1}\lambda_1(M_{\rm thick})\leq \lambda_1(M)\leq
b\log {\rm vol}(M_{\rm thin})\lambda_1(M_{\rm thick})\]
for every closed hyperbolic 3-manifold $M$.
The factor $\log {\rm vol}(M)$ in the statement of the first
part of Theorem \ref{thm1} arises from the ratio
$\lambda_1(M)/\lambda_1(M_{\rm thick})$. 

Although there is
a sequence $M_i$ of hyperbolic
mapping tori of genus $g$ with ${\rm vol}(M_i)\to \infty$
and $\lambda_1(M_i)\geq 
b^\prime \log{\rm vol}(M_i)\lambda_1((M_i)_{\rm thick})$ where
$b^\prime>0$ is another universal constant \cite{H16}, 
we do not know whether such a sequence exists
which moreover satisfies
$\lambda_1((M_i)_{\rm thick})\geq b^{\prime\prime}/{\rm vol}(M_i)^{2^{2g-2}/(2^{2g-2}-1)}$
for a universal constant $b^{\prime\prime}>0$.

Most mapping tori $M$, however, have $\lambda_1(M)$ proportional to 
$1/{\rm vol}(M)^2$. We make this precise in the following result.
Let from now on $S$ be a closed surface of genus $g\geq 2$.

A hyperbolic mapping torus is determined up to isometry by the 
conjugacy class in the \emph{mapping class group} 
${\rm Mod}(S)$ of a defining pseudo-Anosov
element. Conjugacy classes in ${\rm Mod}(S)$ can be listed according
to their translation length. Call a property ${\mathcal P}$ for hyperbolic mapping 
tori \emph{typical} if the proportion of the number of conjugacy classes
of pseudo-Anosov
elements of translation length at most $L$ which give rise to a 3-manifold
with this property tends to one as $L\to \infty$. We refer to 
Section \ref{typical} for a more detailed
discussion. We then say that a typical mapping torus has property
${\mathcal P}$. 

Similarly, we say that a random mapping torus has property
${\mathcal P}$ if a statistical
point for a random walk on ${\rm Mod}(S)$ 
induced by a probability measure on ${\rm Mod}(S)$ whose 
finite support generates
all of ${\rm Mod}(S)$ defines a mapping
torus with this property. Answering a question of Rivin \cite{Rivin}
we show

\begin{thm}
\label{thm2}
For every $g\geq 2$ there is a constant $C_2=C_2(g)>0$ 
so that the following holds true.
Let $M$ be a typical or random mapping torus of genus $g$; then 
\[\lambda_1(M)\leq \dfrac{C_2}{{\rm vol}(M)^2}.\]
\end{thm}

The proof of Theorem \ref{thm2} for random mapping tori
uses the groundbreaking work of Minsky \cite{EL1}
and Brock, Canary and Minsky \cite{EL2} and recurrence properties of 
random walks on the mapping class group
${\rm Mod}(S)$ acting on 
\emph{Teichm\"uller space}
${\mathcal T}(S)$ which 
are of independent interest. In the  
formulation of our main result on random walks, we use the 
following notation. 
For a pseudo-Anosov mapping class $\phi$, we denote by $\ell(\phi)$ 
the translation length of $\phi$ for its action on 
${\mathcal T}(S)$ (which coincides with the translation length
on its axis $\gamma_\phi$). For a number $\zeta>0$ and a subset
$U$ of ${\mathcal T}(S)$ let moreover $N_{\zeta}(U)$ be the 
$\zeta$-neighborhood of $U$ with respect to the Teichm\"uller distance.
We prove.

\begin{thm}\label{thm3}
There exists a number $\zeta=\zeta(g)$ with the following property.  
Let $\mu$ be a nonelementary finitely supported probability measure on the mapping class group.
Let $U \subset \T(S)$ be an ${\rm Mod}(S)$ invariant open subset which 
contains the axis of at least one pseudo-Anosov element.
Then for each $p>0$,
there exists $c=c(U,p)>0$ such that
\begin{align} \mu^{*n}
\{\phi \in {\rm Mod}(S)& \mid
\phi \mbox{ is p-A and }\notag\\
 l(\phi)^{-1}|\{t\in [0,l(\phi)) & : \gamma_{\phi}(t-p, t+p)
\subset N_{\zeta}U \}\vert >c\} \to 1\quad (n \to \infty)\notag\end{align}
\end{thm}

Different but related recurrence properties for axes of random pseudo-Anosov elements have been obtained independently at the same time by Gadre and Maher in \cite{Gadre-Maher}. 

The proof of this result rests on a 
technical tool (Proposition \ref{axisrayteich}) 
which states that for typical trajectories of the random walk, axes of pseudo-Anosov elements in Teichm\"uller space fellow travel rays from a basepoint.
We refer to \cite{Horbez-Dahmani} to closely related earlier work.

The organization of the paper is as follows. In Section \ref{thick} we use the 
work 
\cite{EL1,EL2} of Minsky and Brock, Canary and Minsky to determine 
a collection of graphs with the property that the thick part of every
hyperbolic mapping torus of genus $g$ is uniformly quasi-isometric to a graph in the collection.

 Section \ref{arrays} is devoted to estimating the first eigenvalues of these graphs as a function of their volume. 
 By the main result of \cite{Man05}, the smallest 
 positive eigenvalue of the thick part of a mapping torus
 with Neumann boundary conditions can  be estimated in the same way. 
Theorem \ref{thm1} follows from this fact and \cite{H16} as explained in 
Section \ref{smallest}.
The proof of Theorem \ref{thm2} for typical mapping tori
is contained in  
Section \ref{typical}. Section \ref{random} is devoted to
studying geometric properties of random walks on the mapping class group, with 
the proof of
Theorem \ref{thm3} as the main goal.

 \bigskip
 
 {\bf Acknowledgement:} We are all very grateful to Juan Souto for
helpful discussions. A version of 
Theorem \ref{thm1} is due independently to  
Anna Lenzhen and Juan Souto  \cite{LS16}.
We are also grateful to Samuel Taylor for pointing out a gap in the proof of Proposition 5.1 of a previous version of the paper.

%%%%%%%%%%%%%%%%%%%%%%%%%%%%%%%%%%%%%%%%%%%%%%%%%%%%%%%%%%%%%%%%%%%%%%%%%%%
%%%%%%%%%%%%%%%%%%%%%%%%%%%%%%%%%%%%%%%%%%%%%%%%%%%%%%%%%%%%%%%%%%%%%%%%%%%

\section{The thick part of a mapping torus}\label{thick}

A closed 
hyperbolic 3-manifold $M$ admits a \emph{thick-thin decomposition}
\[M=M_{\rm thin}\cup M_{\rm thick}.\]
The thin part $M_{\rm thin}$ is the set of all points $x$ with injectivity
radius ${\rm inj}(x)\leq \epsilon$ where 
$\epsilon>0$ is sufficiently small
but fixed, and $M_{\rm thick}=\{x\mid {\rm inj}(x)\geq \epsilon\}$.
For an appropriate choice of $\epsilon$, 
$M_{\rm thick}$ is not empty and connected, and  
$M_{\rm thin}$ is a union of (at most) finitely many 
\emph{Margulis tubes}. Such a Margulis tube is
diffeomorphic to a solid torus, and it is a tubular neighborhood
of a closed geodesic of length at most $2\epsilon$. This geodesic
is called the \emph{core curve} of the tube.

The goal of this section is to establish an understanding
of the geometric shape of the thick part of a hyperbolic mapping torus $M$ of
genus $g$. To such a mapping torus $M$,
Minsky \cite{EL1} associates 
a combinatorial model which is quasi-isometric
to $M$. We use this model 
to construct a graph which is $L$-quasi-isometric to 
$M_{\rm thick}$ for a number $L>1$ only depending on $g$.
These graphs will be used in Section \ref{arrays} and Section \ref{smallest} 
for the proof of Theorem \ref{thm1}.  

Furthermore, under some additional
assumption on $M$, we construct geometrically controlled 
submanifolds in $M_{\rm thick}$ with boundary.
These submanifolds will be used to 
estimate the smallest positive eigenvalue
of random mapping tori.

The results in this section heavily depend on the results in 
\cite{EL1,EL2} of Minsky and Brock, Canary and Minsky. The reader who
is not familiar with the ideas developed in \cite{EL1,EL2} will 
however have
no difficulty to understand the statement of Proposition \ref{model}
which is all what is needed for the proof of Theorem \ref{thm1}.

We begin with introducing the class of graphs we are interested in.
By a graph we always mean a finite connected graph $G$. 
We equip $G$ with a metric so that each edge of $G$ has
length one. An \emph{arc} in a graph $G$ is a connected subgraph 
of $G$ which is homeomorphic to an interval. 
The \emph{length} of the
arc is the number of its edges. The length of an arc is at least one.
If $a$ is an arc of length $k$ then $a$ contains $k-1$ vertices of
valence two and two \emph{endpoints} which are vertices of valence one.
A \emph{circle} is a finite connected graph $L$ with all vertices of valence two.
Then $L$ is homeomorphic to $S^1$. Its length equals the number of 
its edges. We always assume that the length of a circle is at least two.
We say that a subgraph $G_1$ of a graph 
$G$ is \emph{attached} to a subgraph
$G_2$ of $G$ at a vertex $v$ if $G_1\cap G_2=\{v\}$.

\begin{definition}\label{arrayofcircles}
For $h\geq 1$,  
an  \emph{array of circles of depth at most $h$} is
a finite connected graph $G$ of the following
form. $G$ contains a subgraph $L$ which is a circle called
a \emph{base circle}. 
If $h=1$ then $G=L$. 
Otherwise $G$ is obtained from $L$ by attaching to each 
vertex of $L$ an array of circles of depth at most $h-1$. 
The \emph{depth} of an array of circles $G$ is defined
to be the smallest number $h$ so that $G$ is of depth at most $h$. 
\end{definition}

%\begin{figure}[h]
%\begin{center}
%\includegraphics[scale=0.4]{array.pdf}
%\caption{An array of circles with the base circle $L$ of depth 3}
%\label{fig:array}
%\end{center}
%\end{figure}

%Figure \ref{fig:array} shows an example of an array of circles. 

Note that a bouquet of two circles is an array of circles of depth two,
but both circles may be used as the base circle, so the base circle
may not be uniquely determined and hence 
a given graph may admit more than one description as an array of circles.
In the sequel, whenever we speak of an array of circles, we 
assume that one choice of such a description has been made.

Closely related to arrays of circles is a more general class of
graphs which we call   
\emph{generalized
arrays of circles}. These are finite connected graphs whose
construction is by induction on a notion of depth $h$ as follows.

If $h=1$ then $G$ is simply a circle. 
In the case $h\geq 2$ we begin as before with a circle
$L$. Given a vertex $v$ of $L$, we allow to either 
attach to $v$ a generalized
array of circles of depth at most $h-1$, or 
we allow to replace $v$ by a graph consisting
of $2\leq s\leq h$ arcs 
$a_1,\dots , a_s$ of possibly distinct length with disjoint interior and
with the same pair of distinct endpoints. 
If $v_1\not=v_2$ are these endpoints, then the graph 
obtained by identifying $v_1$ and $v_2$ 
 is just the base circle $L$ with $s$ circles attached at $v$.
We call the arcs $a_1,\dots,a_s$ \emph{vertex arcs}, and 
we call the vertex $v$ of $L$ 
which was replaced by $a_1,\dots , a_s$ in this 
way a \emph{blown-up vertex}. 

%Figure \ref{fig:blown-up} shows such a 
%blown-up vertex. 

%\begin{figure}[h]
%\begin{center}
%\includegraphics[scale=0.7]{blown-up-vertex.pdf}
%\caption{A vertex $v$ is blown up to the vertex arcs $a_1, \ldots, a_s$}
%\label{fig:blown-up}
%\end{center}
%\end{figure}

We require furthermore that 
for each blown-up vertex with corresponding set 
$a_1,\dots,a_s$ of vertex arcs, there is 
a decomposition $h=\sum_{i=1}^sm_i$ where $m_i\geq 1$. By induction,
we allow to attach to each interior vertex of an arc $a_i$ a generalized array
of circles of depth at most  
$m_i-1$. This also includes the possibility that this interior vertex
is blown up to $u\leq m_i$ arcs with the same endpoints as described above.

As an example, if the depth of the generalized 
array of circles $G$ equals two then
$G$ is obtained from the base circle $L$ by either attaching to 
a vertex of $L$ a (possibly trivial) circle or 
by replacing the vertex by two arcs of
possibly different 
length, and these possibilities are mutually exclusive.

\begin{proposition}\label{model}
There is a number $L=L(g,\epsilon)>0$ with the following property.
Let $M$ be a hyperbolic mapping torus of genus
$g$. Then $M_{\rm thick}$ is $L$-quasi-isometric to a generalized
array of circles of depth at most $2g-2$.
\end{proposition}
\begin{proof}
Let $\hat M$ be the infinite cyclic cover of $M$ 
defined by the fibration $M\to S^1$. 
Its deck group is generated by a 
pseudo-Anosov diffeomorphism $\phi:S\to S$ 
whose mapping torus is $M$.
Fix a homotopy equivalence $S\to \hat M$.  

Part of the main result of \cite{EL1,EL2} can be summarized
as follows.

There exists a \emph{model manifold} $N$ for $\hat M$ which is
homeomorphic to $S\times \mathbb{R}$ and is composed of
combinatorial pieces called \emph{blocks}. This model 
manifold admits an infinite
cyclic group of homeomorphisms compatible with the block
decomposition which is 
generated by a homeomorphism
$\psi:N\to N$. The quotient $N/<\psi>$ is homeomorphic to $M$.

Within $N$ there is a $\psi$-invariant 
subset ${\mathcal U}$ which consists
of open solid tori of the form $U=A\times J$, where
$A$ is an annulus in $S$ and $J$ is an interval in $\mathbb{R}$.
The manifold $N$ is equipped with a $\psi$-invariant
piecewise smooth Riemannian
metric. The induced metric on the boundary 
$\partial U$ of each $U\in {\mathcal U}$ is flat. 
The geometry of the flat torus $\partial U$ is described
by a coefficient $\omega_N(U)\in {\bf H}^2$, where ${\bf H}^2$ is
thought of 
as the Teichm\"uller space of the two-torus (i.e., the space of marked flat metrics on the two-torus).
For $k\geq 1$ 
let ${\mathcal U}[k]$ denote the union of the components of
${\mathcal U}$ with $\vert \omega_N\vert \geq k$ and let
$N[k]=N-{\mathcal U}[k]$. 

The following statement is a combination of 
the Lipschitz Model Theorem and the
Short Curve Theorem as stated in the introduction
of \cite{EL1}, and the Bilipschitz Model Theorem from 
Section 8 of \cite{EL2}.

There exist numbers $K,k>0$ only depending on 
the genus of $S$ but not on the mapping torus $M$, and
there is a $\psi-\phi$-equivariant 
$K$-Lipschitz map $F:N\to \hat M$ with the following
properties.
\begin{enumerate}
\item $F$ induces a marked isomorphism
$\pi_1(N)=\pi_1(S)\to \pi_1(\hat M)$, is proper and has degree one.
\item $F$ is $K$-bilipschitz on $N[k]$ with respect to the 
induced path metric.
\item $F$ maps each component of ${\mathcal U}[k]$ to a Margulis
tube, and each Margulis tube with sufficiently short core curve 
is contained in the image of 
a component of ${\mathcal U}[k]$.
\end{enumerate}

Thus all we need to show is that for the number $k$ in the above
statement, 
$N[k]$ is $L$-quasi-isometric
to a generalized array of circles for a universal number $L>0$.
This statement in turn follows from the construction of the
model manifold which we outline next (we refer to \cite {EL1} for 
all details).

A \emph{clean marking} of the surface $S$ consists of a 
pants decomposition $P$ of $S$,
the so-called \emph{base} of the marking, 
and a set of so-called \emph{spanning curves}. For each pants curve
$c\in P$, there exists a unique spanning curve. This spanning curve is contained
in $S-(P-c)$, and it intersects $c$ transversely in one
or two points depending on whether the component of $S-(P-c)$ containing $c$ is 
a one-holed torus or a four-holed sphere. 

A pants decomposition $P$ of $S$ is \emph{short} in $M$ 
if there is a map $F:S\to M$ in the given homotopy class 
which maps each component of $P$ to a geodesic
in $M$ of uniformly bounded length.

To build the model manifold, start with a pants decomposition $P$
which is short in $M$. Since each point in $M$ is uniformly
near a \emph{pleated surface} $f:(S,\sigma)\to \hat M$ 
(see Theorem 3.5 of \cite{Minsky93} for this result of Thurston),
short pants decompositions exist. 
Namely, such a pleated surface $f$ is a path isometry for 
a hyperbolic metric $\sigma$ on $S$. Furthermore, for every hyperbolic
metric on $S$ 
there is a pants decomposition of uniformly bounded length.

A pants decomposition of $S$ can be viewed as a maximal
simplex in the \emph{curve complex} ${\mathcal C}(S)$ of $S$.
By Theorem 6.1 and 
Theorem 7.1 of \cite{EL1}, we may assume that
a short pants decomposition in $M$ is a simplex in  
${\mathcal C}(S)$ which is uniformly
near (for the distance in ${\mathcal C}(S)$) to 
a curve in any choice of a \emph{hierarchy} constructed
from the \emph{ending laminations} of $\hat M$. These
ending laminations  
are just the supports of the horizontal and 
vertical measured geodesic laminations on $S$ which determine
the axis of the pseudo-Anosov mapping class $\phi$. 

The vertical measured geodesic lamination $\lambda$ of the 
axis of $\phi$ determines 
an essentially unique
clean marking $\mu$ of $S$ with base the given 
pants decomposition $P$. The spanning curves are
determined by the \emph{subsurface projection}
of $\lambda$ into the collars of $P$.
We refer to \cite{MM00,EL1} for details of this construction.

Let $\phi(\mu)$ be the image of $\mu$ under $\phi$. To
$\mu$ and $\phi(\mu)$ we can associate a \emph{hierarchy}  
$H$ and a \emph{resolution} of $H$. This hierarchy 
consists of a collection of so-called \emph{tight
geodesics} in the curve complex of connected subsurfaces of $S$
different
from three-holed spheres. 
The hierarchy is required to be \emph{four-complete}.
This means the following. 

Define the \emph{complexity} $\xi(Y)$ of a connected 
subsurface $Y$ of $S$ of genus $h\geq 0$ with 
$b\geq 0$ boundary components by
$\xi(Y)=3h+b$. Suppose that $Y\subset S$ 
is a complementary component of a vertex in a geodesic $h$ 
from the hierarchy $H$ whose domain is a surface
$Y^\prime\supset Y$. If $\xi(Y)\geq 4$ then 
$Y$ is the domain of a geodesic in $H$.

%a path of clean markings connecting $\mu$ to 
%$\phi(\mu)$. Two adjacent markings in this path differ
%by an \emph{elementary move} (see 
%\cite{MM00} and Section 5.4 of \cite{EL1} for details)
%which alters the marking either by performing 
%a single Dehn twist 
%about a curve in the base of the marking 
%(which only alters a single spanning curve) 
%or by a ``switch move''
%whose support (a so-called \emph{component domain} of the 
%resolution) 
%either is a four-holed sphere or a one-holed torus.

Choose such a four-complete hierarchy $H$ 
associated to $\mu$ and $\phi(\mu)$ as well as a resolution
of $H$. Each edge $e$ in a geodesic from the hierarchy
$H$ whose domain $D(e)$ satisfies $\xi(D(e))=4$
(i.e. $D(e)$ either is a four-holed sphere or a one-holed torus) 
%Each elementary move in the resultion 
%can be thought of as an edge $e$ in the 
%marking graph of the component domain $D(e)$ which is
%the support of the move. 
%Such an edge $e$ 
defines a \emph{block} $B(e)$ for
the \emph{component domain} $D(e)$. The backward endpoint
$e_-$ of $e$ and the forward endpoint $e_+$ can be identified
with a simple closed curve in the component domain $D(e)$ of
distance one in the curve graph of $D(e)$.

 The block $B(e)$ for the edge $e$ is then defined as
\[B(e)=(D(e) \times [-1,1]) - 
({\rm collar}(e_-)\times [-1,-1/2)\cup {\rm collar}(e_+)\times (1/2,1]).\]
The \emph{glueing boundary} of the block $B(e)$ is defined to be
\[\partial_{\pm}B(e)=(D(e)-{\rm collar}(e^{\pm}))\times \{\pm 1\}.\]
This glueing boundary is a union of three-holed spheres.

There are only two combinatorial types of blocks \cite{EL1}. 
Each block can be equipped with a standard Riemannian metric
with totally geodesic boundary 
so that combinatorially equivalent blocks are isometric. 
The blocks are glued along the components of their glueing boundaries
as prescribed by the resolution of the hierarchy $H$ and such that 
 the metrics on the glueing boundaries of 
the blocks match up. 
Let $N[0]$ be these glued blocks (compare \cite{EL1} for notation). 
Then $N[0]$ is 
a Riemannian manifold whose boundary is a union of two-dimensional
tori.

Each boundary torus $\partial U$ contains a distinguished free homotopy class of
simple closed geodesics, the \emph{predicted meridians}.
Such a predicted meridian is given as follows. There is a simple closed
curve $v$ on $S$ so that for any simple arc $a$ on $S$ connecting
the boundary components of a collar neighborhood of $v$, the predicted
meridian equals $\partial(a\times [s,t])$ where the parameters
$s<t$ can be read off from the hierarchy (we refer to p.80 of \cite{EL1} for
details). 
This predicted meridian and the simple closed curve $v$ of length
$t=\epsilon$ (here as before, $\epsilon >0$ is a fixed Margulis constant) 
determine $\partial U$ as a marked flat torus. 
Following Section 3.2 of \cite{EL1}, these data also determine uniquely
a \emph{meridian coefficient} 
$\omega_N(\partial U)\in {\bf H}^2$.  
The length of the predicted meridian of the torus equals
$\epsilon \vert \omega_N(\partial U)\vert$,
and the imaginary
part $\Im \omega$ equals $1/\epsilon$ times 
the sum of the heights of the
annuli that make up $\partial U$ (see p.80 of \cite{EL1}).

Glue solid tori to those boundary tori $\partial U$ with coefficient
$\vert \omega_N(U)\vert \leq k$. Up to isotopy, 
there is a unique way of such a glueing
which maps the meridian of the solid torus to the predicted meridian
of the boundary torus
(compare again \cite{EL1} for details). 
The resulting manifold $N[k]$ 
is the model for $M_{\rm thick}$.
Thus we have to  
verify that indeed, $N[k]$ is $L$-quasi-isometric to a
generalized array of circles of depth at most $2g-2$ 
for some $L>0$ only depending on $g$. Since the diameters
of the tubes in $N[k]-N[0]$ are uniformly bounded, 
for this it suffices to show that $N[0]$ is uniformly
quasi-isometric to a generalized array of circles  
of depth at most $2g-2$.

The pseudo-Anosov map $\phi$ maps
the marked surface
$(S,\mu)$ to $(S,\phi(\mu))$. By equivariance, there is a distinguished
\emph{main tight geodesic} $g_H$ in the hierarchy whose 
\emph{domain} is
the surface $S$ and which connects the marking $\mu$ to 
$\phi(\mu)$.
This tight geodesic consists of a sequence 
of simplices $(v_i)$ in the curve complex
${\mathcal C}(S)$ of $S$ so that for any vertices 
$w_i$ of $v_i$, $w_j$ of $v_j$ we have
$d_{{\mathcal C}_1(S)}(w_i,w_j)=\vert i-j\vert$ (here 
$d_{{\mathcal C}_1(S)}$ is the distance function of the one-skeleton
of ${\mathcal C}(S)$). 
Moreover,
for all $i$, $v_i$ represents the boundary of the subsurface
of $S$ filled by $v_{i-1}\cup v_{i+1}$. 
 
Glue the two ends of the tight geodesic $g_H$ 
using the map $\phi$ and view
it as the base circle $L$ of a generalized array of circles.
The length of 
$L$ equals the length of the tight geodesic $g_H$. 
   
We use the resolution of the hierarchy to construct from $L$ a
generalized array of circles as follows.           
A vertex $v_i$ of $g_H$ 
decomposes $S$ into complementary regions. Such a region $Y$ 
is a \emph{component domain} of the hierarchy.
If $Y$ is a component of $S-v_i$ different from a three-holed sphere
then $v_{i-1}\vert Y$ and 
$v_{i+1}\vert Y$ are either markings of $Y$ or empty
since $v_i$ is the boundary of
the subsurface filled by $v_{i-1}\cup v_{i+1}$ (we ignore here
the modification needed for the first and last simplex). 
In the case that $v_{i-1}\vert Y$ and 
$v_{i+1}\vert Y$ are both markings of $Y$,
the hierarchy contains a tight geodesic with domain $Y$ connecting
these markings. 

As an example, if $Y$ is a connected subsurface of $S$ and if 
the \emph{subsurface
projection} \cite{MM00} 
of $\mu\cup \phi(\mu)$ into $Y$ (which is never empty
since $\mu$ and $\phi(\mu)$ are markings of $S$)
has large diameter, then $Y$ arises as a component domain
in the hierarchy. If $v_i$ 
is the simplex in the hierarchy with $Y$ as component
domain, then the 
adjacent simplices $v_{i-1}$ and $v_{i+1}$ 
in the geodesic of the hierarchy
containing $v_i$ fill $Y$ \cite{MM00}. 
Thus there is a geodesic in the 
hierarchy with domain $Y$, and the
length of this geodesic coincides 
with the diameter of this subsurface projection up to a universal
additive constant (Lemma 5.9 of \cite{EL1} summarizes this 
result from \cite{MM00}).
Even if
$v_{i-1}$ and $v_{i+1}$ do not intersect the component 
$Y$ of $S-v_i$, if $Y$ is different from a three holed sphere then 
there is a geodesic in the hierarchy whose domain equals $Y$. 
As this will not be important for our purpose, we refer 
to \cite{EL1} for a discussion of these technicalities.
%Such a geodesic 
%consists of a chain of elementary moves, forming a uniform
%quasi-geodesic in the marking complex of $Y$.

For each $i$ let $u_i$ be the vertex of the circle 
$L$ corresponding to the simplex $v_i$.
Let $Y_1,\dots,Y_s$ be the complementary regions of $v_i$ 
in $S$. Let $a_i$ be the length of the tight
geodesic in the hierarchy 
with domain $Y_i$. If $Y_i$ is a three-holed sphere
then we define $a_i=1$. 
Blow up
the vertex $u_i$ and replace it by $s$ arcs
of length $a_i$. Moreover, associate to the arc $a_i$
the absolute value $-\chi(Y_i)$ 
of the Euler characteristic of $Y_i$. 
Note that $-\sum \chi(Y_i)=-\chi(S)=2g-2$.

Each arc $a_i$ corresponds to a tight geodesic
in a surface $Y_i$ of Euler characteristic $\chi(Y_i)$. 
Repeat the above construction with these arcs, 
successively blowing up vertices. Inductively, this defines
a generalized array of circles of depth at most $2g-2$.

To summarize, from the mapping torus we obtain
(non-uniquely) a hierarchy $H$ and a resolution of $H$. 
The resolution is used to construct
a generalized array of circles $G$ of depth at most $2g-2$. 
There is natural
map $\Psi:N[0]\to G$ which maps a block in $N[0]$ to an
edge of $G$. 

We are left with showing that this generalized array of circles
indeed is uniformly quasi-isometric to $N[0]$. To this
end simply recall that in the situation at hand, there are only
two types of blocks \cite{EL1}. The first type of blocks
is obtained from a component domain 
$D$ which is a one-holed torus.  
In the construction of the generalized array $G$, 
a geodesic $\eta$ in the hierarchy $H$ with a one-holed torus as 
component domain gives rise to an
outmost arc, i.e. an arc of biggest depth. 
In the model manifold, it corresponds to a chain of blocks
whose length equals the length of $\eta$. 

In the natural order of blocks in the chain given by
an orientation of the Teichm\"uller geodesic which defines
the mapping torus, 
the top component
of the glueing boundary of the 
last block in the chain consists of one three-holed sphere.
This sphere is glued to the glueing boundary of a block $B$ 
which arises from a different geodesic of the hierarchy.
The block $B$ is of the second type, obtained from a
component domain which is a four-holed sphere. 
Then the second three-holed sphere in the
glueing boundary of the block $B$ lying on the same ``side'' 
is glued to a block
arising from a different geodesic in the hierarchy. This 
three-holed sphere may already
be present at the initial point of the geodesic.

In the generalized
array of circles, this corresponds precisely to glueing the endpoints
 of the arc representing the geodesic to the endpoints of another
 arc.
 The discussion of geodesics in $H$ whose component domains are four-holed
 spheres is completely analogous and will be omitted. 
 
 As a consequence, the map which associates to a block in $N[0]$ 
 the edge of the generalized array of circles $G$ corresponding to it
 is essentially the map which associates to the decomposition
 of $N[0]$ into blocks the dual graph. The deviation from this
 precise picture comes from the addition of some additional
 edges in $G$, one for each "reassembling point", to meet the
 requirement of a generalized array of circles which we found
 most useful for our purpose. 
 The proposition is proven.
\end{proof}

\begin{remark}\label{nonseparating}
The above construction also yields the following. Assume that for some
$\delta>0$, 
$M$ is a mapping torus defined by a pseudo-Anosov mapping class
whose axis is entirely contained in the 
$\delta$-thick part of Teichm\"uller space.
Then $M$ is $L(\delta)$-quasi-isometric to a circle where
$L(\delta)>0$ only depends on $\delta$. 

More generally, let $M$ be 
a hyperbolic mapping torus of 
genus $g$ which is defined by a pseudo-Anosov mapping class
$\phi$ of translation length $\ell$ in Teichm\"uller space.
Let us assume that for some fixed
number $c_1\in (0,1)$, 
the translation length for the action of 
$\phi$ on the curve graph of the surface of genus $g$
is at least $c_1\ell$. Then the length of the base circle of the
array of circles constructed from $M$ in the proof of
Proposition \ref{model} is at least 
$c_2{\rm vol}(M_{\rm thick})=c_3{\rm vol}(M)$  
where $c_2,c_3>0$ only depend on $c$ and $g$ (compare
the discussion in \cite{H16} for a comparison between
${\rm vol}(M_{\rm thick})$ and ${\rm vol}(M)$).
\end{remark}

\begin{remark} \label{tubecoefficients}
From the model and the construction of a generalized
array of circles $G$, we obtain some information of 
the size of the Margulis tubes in the mapping torus $M$. Namely, 
Margulis tubes with boundary of large volume correspond
to blown-up vertices in the construction of $G$ with 
at least one long vertex arc. 

However, large subsurface projections into the complement
of a non-separating simple closed curve on $S$ give
rise to Margulis tubes which can not be detected in the
generalized array of circles. Thus the thick-thin
decomposition of $N$ can not be read off from the 
generalized array of circles.

As a consequence, a base circle of length
$\sim {\rm vol}(M)$ does not imply that translation
length for the action of the corresponding pseudo-Anosov
element on the curve graph 
is proportional  to its translation length on 
Teichm\"uller space.
\end{remark}

For the proof of the second part of Theorem \ref{thm1} we construct
collections of mapping tori with fibre genus $g$ which are uniformly
quasi-isometric to specific generalized arrays of circles. 

We first introduce the arrays of circles we are interested in.
Namely, let $G$ be an array of circles with base circle $L$ and depth $h$.
Define the \emph{depth} of a circle $C$ in $G$ 
as the minimal depth of an array of circles $G^\prime\subset G$ with base circle $L$ which 
contains $C$. Thus each circle in an array of circles 
of depth $h$ has depth at most $h$, and the base circle is the unique 
circle of 
depth one.

Call an array of circles $G$ \emph{step-homogeneous} if 
all circles in $G$ of the same depth are non-degenerate
and of the same length. 
A special example of a step-homogeneous array of circles is 
an array where for some $k\geq 2$, the circles of 
depth $\ell$ have length $k^{2^{\ell-1}}$. We call this 
array \emph{optimal}. Note that an optimal array of circles
is uniquely determined by the length $k$ of its base circle
and by its depth.

\begin{proposition}\label{example}
For each $g>0$ there is a number 
$c_1(g)>0$ with the following properties.
For $k\geq 2$ let $G$ be an 
optimal step-homogeneous array of circles  
of depth $2g-2$, with base circle of length $k$.
Then there is a mapping torus
$M$ of genus $g$ so that $M_{\rm thick}$ is $c_1(g)$-quasi-isometric to $G$.
\end{proposition}
\begin{proof}
Choose a decomposition of $S$ as a descending sequence of 
connected subsurfaces $S=S_0\supset S_1\supset \dots \supset S_{2g-3}$ 
with the following properties. 
\begin{enumerate}
\item $\chi(S_i)=2-2g+i$.
\item $S_{2g-3}$ is a one-holed torus.
%\item For each $i$, $S_i$ is a union of components of $S-P$.
\end{enumerate}
Such a chain can for example be constructed as follows.
Choose a simple closed curve $\alpha_1$ which decomposes $S$ into a one-holed
torus and a surface $S_1$ of genus $g-1$ with connected boundary.
Choose a pair $(\alpha_2,\alpha_2^\prime)$ of simple closed curves which decompose
$S_1$ into a three-holed sphere and a surface $S_2$ of genus $g-2$ with 
two boundary components. One proceeds inductively by decomposing $S_i$ into a three-holed sphere and 
a surface $S_{i+1}$ until $S_{i+1}$ becomes a one-holed torus. Since a one-holed torus has Euler characteristic -1, 
the chain ends with the index $2g-3$. Let ${\mathcal C}(S_i)$ be the curve complex of $S_i$. 

The union of the boundary circles of the surfaces $S_i$ 
is a pants decomposition $P$
of $S$. Let $\tau$ be a \emph{train track} in standard form for $P$.
We can choose $\tau$ in such a way that 
it restricts to a train track in standard form on each of 
the subsurfaces $S_i$ (here we have to be a bit
careful what this means as $S_i$ has boundary). For 
terminologies regarding train tracks, 
we refer the readers to \cite{PH92}. In particular, one can find 
the definition of train track in standard form
for a pants decomposition of $S$ and the proof 
of its existence in Sections 2.6 and 2.7 of \cite{PH92}.

For a subsurface $Y$ of $S$, recall that the mapping class
group ${\rm Mod}(Y)$ of $Y$ consists of isotopy classes
of diffeomorphisms of $Y$ which fix the boundary of $Y$ pointwise.
For each of the subsurfaces $S_i$ $(0\leq i\leq 2g-3)$
choose once and for all 
a pseudo-Anosov mapping class $\phi_iÂž\in {\rm Mod}(S_i)$ 
with the following properties.
\begin{enumerate}
\item $\phi_i$ admits $\tau\vert S_i$ as a \emph{train track expansion}.
\item $d_{{\mathcal C}_1(S_i)}(c,\phi_i(c))\geq 5$ for every 
simple closed curve $c$ on $S_i$.
\end{enumerate}
Here as before, $d_{{\mathcal C}_1(S_i)}$ is the distance
in the one-skeleton of the curve complex of the
subsurface $S_i$.

The existence of $\phi_i$ satisfying (1) is a consequence of 
the fact that $\tau\vert S_i$ is maximal birecurrent 
which follows from the requirement that 
 $\tau\vert S_i$ is in the standard form 
(as shown in \cite{PH92}). (2) can be easily satisfied by first taking some 
$\phi_i$ satisfying (1) and replacing it by some positive power. 
Namely, by \cite{MM99}, 
for any pseudo-Anosov map $f$ on a non-exceptional surface $\Sigma$ and 
for any simple closed curve $\alpha$ on the surface, 
the limit $\lim_{n \to \infty} \dfrac{d_{\mathcal{C}_1(\Sigma)}(f^n(\alpha), \alpha)}{n}$ exists and positive, 
and it is independent of the choice of $\alpha$. 

Fix a number $k>1$. We define inductively a pseudo-Anosov mapping class
$\Psi_k\in {\rm Mod}(S)$ as follows. For $m\geq 1$ 
write $\ell(m)=k^{{2^{2g-2-m}}}$. Define $\eta_1=\phi_{2g-3}$ and 
inductively let 
\[\eta_m= \phi_{2g-2-m}\circ \eta_{m-1}^{\ell(m-1)}.\]  
Then for each $m$, $\eta_m$ is a pseudo-Anosov diffeomorphism
of $S_{2g-2-m}$. 
The mapping class $\Psi_k=\eta_{2g-2}$ is pseudo-Anosov, with 
train track expansion $\tau$. We refer to Section 6 of 
\cite{Hamenstaedt16} for details of this construction.

Now consider the mapping torus $M_k$ of $\Psi_k$. 
As a hierarchy for $M_k$ is defined by 
singling out subsurfaces $Y$ of $S$ so that the subsurface
projections of the vertical and horizontal measured geodesic
laminations of the axis of $\Psi_k$ is large, 
it follows from Proposition \ref{model} (and its proof) that
$(M_k)_{\rm thick}$ is $L$-quasi-isometric to an optimal
step-homogeneous array
of circles of base length $k$ for a constant $L>1$ not
depending on $k$.

To give a more detailed account on this fact, 
the main geodesic of the hierarchy
corresponds to a fundamental domain in a quasi-axis in ${\mathcal C}(S)$
for the pseudo-Anosov mapping class $\phi_0^k$. This main geodesic
determines  
the base circle $L$ of the generalized
array of circles in the construction from the proof of 
Proposition \ref{model}. The length of this base circle  
is uniformly equivalent to $k$.

In a second step, the hierarchy contains geodesics in the curve graph of 
copies of the surface $S_1$ whose length is equivalent to $k^2$.
In the construction of the generalized array of circles, this
amounts to blowing up each vertex $v$ of $L$ and replacing it by
a single edge and an arc of length equivalent to $k^2$ with the
same endpoints.  
Let $G_1$ be the resulting graph. 
Contraction of each vertex arc in $G_1$ which consists of a single edge
yields an array of circles which is 
uniformly quasi-isometric to $G_1$. This array of circles 
is uniformly quasi-isometric to an optimal array of circles 
of depth two. 
By induction, we conclude that
indeed, the thick part of $M_k$ is uniformly quasi-isometric to an optimal
array of circles of depth $2g-2$. The proposition follows. 
\end{proof}

We use Proposition \ref{model} and its proof to
obtain some additional geometric information on doubly degenerate
hyperbolic 3-manifolds which are used in the proof of Theorem \ref{thm2}. 

Let $\gamma\subset {\mathcal T}(S)$ be a bi-infinite Teichm\"uller geodesic
which defines a doubly degenerate hyperbolic 3-manifold $M$ with 
filling end invariants. Suppose that for some $\epsilon >0$ the geodesic
contains a subarc $\gamma[a,b]$ entirely contained in
the $\epsilon$-thick part  
${\mathcal T}(S)_\epsilon$ of ${\mathcal T}(S)$ of all marked
hyperbolic surfaces of injectivity radius at least $\epsilon$.
By \cite{MM99} (see \cite{Hamenstaedt10} for an
explicit statement), there is a number $\chi=\chi(\epsilon)>0$
so that the map which associates to 
$t\in [a,b]$ a closed geodesic of smallest length on 
the hyperbolic surface $\gamma(t)$ 
is a $\chi$-quasi-geodesic in the curve complex
${\mathcal C}(S)$ of $S$. Therefore 
the endpoints $\gamma(a),\gamma(b)$ define
(non-uniquely) a hierarchy $H$ 
all of whose geodesics different from the main geodesic have 
uniformly bounded length
\cite{MM00}. Moreover, if $b-a$ is sufficiently large
then the length of the main geodesic is larger than
any prescribed threshold. 

Let $v$ be any vertex in the main geodesic associated
to $\gamma[a,b]$. 
We are only interested in vertices not too close
to the endpoints of the hierarchy $H$. Such a vertex is a multicurve in $S$
whose length becomes short along $\gamma[a,b]$, say at $\gamma(s)$.  By 
Lemma 7.9 of \cite{EL1}, the lengths of the closed geodesics in $M$
in the free homotopy classes  of the components of $v$ is uniformly bounded. 
There is a pleated surface $f:S\to M$ which maps the curves from 
a maximal simplex $\Delta\subset {\mathcal C}(S)$ with
$v\subset \Delta$ to geodesics in $M$, and these geodesics all have
moderate length in $M$.
Moreover, the set of simple closed curves on the pleated surface which 
have uniformly bounded length in $M$ is contained in the $d^\prime$-neighborhood 
of $v$ in ${\mathcal C}(S)$ 
for some universal number $d^\prime >0$.

By the tube penetration Lemma 7.7 of \cite{EL1}, 
there is a number $r>0$ only depending on $\epsilon$ and $g$ 
with the following property. Let $s\in [a+r,b-r]$ and let $v\in H$ 
be a vertex corresponding to a multicurve which becomes
short for $\gamma(s)$; then the diameter of a
pleated surface mapping $v$ geodesically is uniformly bounded.

Theorem 6.2 of \cite{EL2} now shows that up to enlarging $r$, such a pleated
surface $f:(S,\sigma)\to M$ can be deformed with a homotopy to an 
embedding $F:(S, \sigma)\to M$ with the following properties.
\begin{enumerate}
\item $F(S)$ is contained in the $r$-neighborhood of $f(S)$.
\item The second derivatives of $F$ 
are uniformly bounded. 
\end{enumerate}

We use this to show

\begin{proposition}\label{control}
For every $\epsilon >0$ there exists 
a constant $c_2=c_2(g,\epsilon)>0$ with
the following property. 
Let $\hat M$ be a doubly degenerate hyperbolic 3-manifold
which is an infinite cyclic cover of a mapping torus $M$ of genus $g$.
Suppose that the Teichm\"uller geodesic $\gamma$ defining $M$ contains
a segment $\gamma[a,b]\subset {\mathcal T}(S)_\epsilon$ of 
length $b-a\geq 2c_2$. Then $\hat M$
contains a smooth embedded 3-manifold $N_0$ with boundary  $\partial N_0$ with 
the following properties.
\begin{enumerate}
\item $\partial N_0=\Sigma_{a}\cup \Sigma_{b}$, and there are diffeomorphisms 
$f_1:(S,\gamma(a+c_2))\to \Sigma_{a},f_2:(S,\gamma(b-c_2))\to \Sigma_{b}$ whose
derivatives are uniformly bounded.
\item There is a smooth surjective map $N_0\to [a,b]$ of uniformly bounded
derivatives.
\end{enumerate}
If $\gamma[a^\prime,b^\prime]\subset {\mathcal T}(S)_\epsilon$ is another such segment
so that $[a,b]\cap [a^\prime,b^\prime]=\emptyset$ then the corresponding
$3$-manifolds $N_0,N_0^\prime$ are disjoint. 
\end{proposition}
\begin{proof}
The above discussion implies that $M$ contains a submanifold which is 
a union of pieces diffeomorphic to $S\times [0,1]$ whose boundaries
are deformations of pleated surfaces determined by hyperbolic metrics
along $\gamma[a,b]$. The number of these pieces is proportional to 
$b-a$. Disjoint intervals in the parameter space of $\gamma$ as
in the proposition give rise to disjoint pieces. 

The pieces are of uniformly bounded diameter, and their
boundary surfaces have uniformly bounded geometry. 
Then there is a controlled tubular neighborhood of each of these boundary
surfaces, and these neighborhoods can be used to construct the map 
onto $[a,b]$.  
\end{proof}

%%%%%%%%%%%%%%%%%%%%%%%%%%%%%%%%%%%%%%%%%%%%%%%%%%%%%%%%%%%%%%%%%%%%%%%%%%%
%%%%%%%%%%%%%%%%%%%%%%%%%%%%%%%%%%%%%%%%%%%%%%%%%%%%%%%%%%%%%%%%%%%%%%%%%%%

\section{Arrays of circles}\label{arrays}

The main result of 
\cite{Man05} states the following. Let $M$ be a closed Riemannian
manifold of bounded geometry whose injectivity radius is
bounded from below by a fixed positive constant. 
If $M$ is uniformly quasi-isometric to 
a finite graph $G$ then the smallest positive 
eigenvalue of $M$ is uniformly
equivalent to the smallest positive eigenvalue of $G$. This statement
is also valid without modification for compact manifolds $M$ with
boundary and Neumann boundary conditions (see \cite{LS12} for 
a more precise statement).

Let us consider as before a hyperbolic mapping torus $M$ of genus $g\geq 2$.
We showed in Section \ref{thick} that the thick part $M_{\rm thick}$ of $M$ 
is $L=L(g)$-quasi-isometric
to a generalized array of circles of depth at most $2g-2$. 
Thus to estimate the smallest eigenvalue of $M_{\rm thick}$
with Neumann boundary conditions it suffices to
estimate the smallest eigenvalue of a generalized array of circles
of a given depth. The purpose of this section is to establish
such an estimate.

Let for the moment $G$ be any finite connected graph with vertex set
${\mathcal V}(G)$ and edge set ${\mathcal E}(G)$. 
Denote by
${\mathcal F}_0(G)$ the vector space of functions
\[f:{\mathcal V}(G)\to \mathbb{R}\]
with the property that $\sum_vf(v)=0$. 
We equip ${\mathcal F}_0(G)$ with the usual $\ell^2$-inner product
\[(f,h)=\sum_v f(v)g(v).\]

For each such function $f$, the \emph{Rayleigh quotient}
${\mathcal R}(f)$ is defined
by \[{\mathcal R}(f)=\frac{\sum_v\sum_{w\sim v}(f(w)-f(v))^2/p(v)}
{\sum_v f^2(v)}\]
where $p(v)$ is the degree of the vertex $v$, and $v \sim w$ means 
that $v$ and $w$ are connected by an edge. 
The first eigenvalue of $G$ is defined as
\[\lambda_1(G)=\inf\{{\mathcal R}(f)\mid 0\not=f\in {\mathcal F}_0(G)\}.\]

In the sequel we adopt analytic notations, and we write
\begin{equation}\label{first}
\int ( f^\prime)^2=\sum_v\sum_{w\sim v}(f(w)-f(v))^2/p(v)\end{equation}
and $\int f^2=\sum_vf^2(v)$.

Throughout the rest of this section
we view a graph $G$ as a metric space with edges of length one.
Thus the length of a subarc of $G$ equals its combinatorial length,
i.e. the number of its edges.

Our first goal is to 
establish an upper bound for the first eigenvalue 
of an array of circles. In a second step, we then
extend the bound to a generalized array.

We will make use of the
Minmax-principle which is equally valid for the Laplacian on 
manifolds as well as for the Laplacian on graphs. 
For a finite graph $G$ it states the 
following.

Let $\rho_0,\rho_1:{\mathcal V}(G)\to \mathbb{R}$ be any two 
nontrivial functions
with disjoint support; then 
\[\lambda_1(G)\leq \max_{i=0,1}{\mathcal R}(\rho_i).\]

\begin{proposition}\label{graphfinished}
Let $G$ be an array of circles of depth $h$; then
\[\lambda_1(G)\leq 64\pi^2/ {\rm vol}(G)^{2^h/(2^{h}-1)}.\]
\end{proposition}
\begin{proof}
We show by induction on $h$ 
the following. 
Let $G$ be an array of circles of 
depth $h$; then for every vertex $v$ of the base circle of $G$ 
there is a function $f\in {\mathcal F}_0(G)$ with $f(v)=0$ so that
${\mathcal R}(f)\leq  64\pi^2/{\rm vol}(G)^{2^h/(2^{h}-1)}$.

In the case $h=1$, $G$ is a circle and the claim is straighforward.
Thus assume that the 
claim holds true for $h-1$. Let $G$ be an 
array of circles 
of depth $h$, with base circle $L$.
Assume first that there is a vertex $v$ of $L$ so that the volume of the 
descendant $G_v$ of $L$ at $v$ 
(i.e. the array of circles of depth at most $h-1$ attached to $L$ at $v$)
is at least
${\rm vol}(G)^{(2^h-2)/(2^h-1)}$. 

By the 
induction hypothesis, there is a function $f\in {\mathcal F}_0(G_v)$ with 
$f(v)=0$ and such that 
\[\mathcal{R}(f)\leq 64\pi^2/{\rm vol}(G_v)^{2^{h-1}/(2^{h-1}-1)}.\]
Extend $f$ by zero to $G$. The extended function $F$ vanishes
on the base circle $L$ of $G$, and it is contained in 
${\mathcal F}_0(G)$. Moreover, 
\begin{align}{\mathcal R}(F)&\leq 64 \pi^2/({\rm vol}(G)^{(2^h-2)/(2^h-1)})^{2^{h-1}/(2^{h-1}-1)}
\notag\\ & =
64 \pi^2/{\rm vol}(G)^{2^h/(2^h-1)}.\notag\end{align}
Thus the function $F$ satisfies all the requirements in the above claim,
for every vertex of the base circle.

The second case is that the volume of 
every descendant of $L$ is strictly smaller than
\[{\rm vol}(G)^{(2^h-2)/(2^h-1)}=E.\]
Let $\ell\geq 2$ be the length of the base circle $L$,
let $v\in L$ be any vertex and let 
$\alpha:[0,\ell]\to L$ be a simplicial parametrization of 
$L$ by arc length with $\alpha(0)=v$ 
which maps the integral points in 
$[0,\ell]$ to the vertices of $L$.
For $1\leq k\leq \ell$ and 
$t\in [k-1/2,k+1/2)$, let $\eta(t)$ be one plus 
the volume ${\rm vol}(G_{\alpha(k)})$ 
of the descendant $G_{\alpha(k)}$ 
of $L$ at $\alpha(k)$ (with the obvious 
interpretation for $k=\ell$). 
Note that $1\leq \eta(t)\leq E$ for all $t$.

Define 
\[\beta(t)=\int_0^t \frac{1}{E}\eta(s)ds;\] then $\beta$
is differentiable outside the points $k+\frac{1}{2}$ for 
$k\in \mathbb{Z}$, and moreover $0<\beta^\prime(t)\leq 1$ for all $t$. 
More precisely, 
$\beta$ is a piecewise-linear continuous function which is strictly
increasing, and 
\[\beta(\ell)=(\ell + {\rm vol}(G))/E=\ell/E + E^{1/(2^h-2)}=E^\prime.\] 

Let $m\in (0,\ell)$ be such that $\beta(m)=E^\prime/2$. 
For $t\in [0,\ell]$ define a function $f_1$ supported in 
$[0,m]$ by 
\begin{equation}
f_1(t)=\
\begin{cases}
\sin(4\pi \beta(t)/E^\prime) &\text{if $0\leq t\leq m$,}\\
0 &\text{otherwise.}
\end{cases}
\end{equation}
and define similarly a function $f_2$ supported in $[m,\ell]$.
%Via composition with the inverse $\alpha^{-1}$ of 
%the parmetrization $\alpha:[0,\ell]\to L$ we can view $f$ as a continuous 
%function on $L$. 

By the mean value theorem, for each $k$ there exists 
some $t_k\in [k-\frac{1}{2},k+\frac{1}{2}]$ such that
$f_i(t_k)$ satisfies 
\[f_i(t_k)^2=\int_{k-\frac{1}{2}}^{k+\frac{1}{2}}f_i^2(s)ds.\]
Define 
\[F_i(\alpha(k))=f_i(t_k)\] and 
extend $F_i$ to a function on $G$ which is constant on each
of the arrays of circles $G_{\alpha(i)}$ which are attached to the 
vertices of the circle $L$.
Since the function $\eta$ is constant on each of the intervals
$[k-\frac{1}{2},k+\frac{1}{2}]$, we conclude that
\begin{equation}\label{squareintegral}
\int_0^\ell f_i^2(t)\eta(t)dt=\sum_v F_i^2(v).
\end{equation}

%
% and extend
%$f$ to a function which is constant on each descendant of $L$. Morally, one should think of $f$ as a real-valued function defined on 
%${\mathcal V}(G)$ by replacing the point mass at each vertex by a uniform Lebesgue measure on a small interval around the vertex. 
Our strategy now is to show that ${\mathcal R}(F_i)$ is close
to the quotient 
\[{\mathcal R}(f_i)=\int_0^\ell (f_i^\prime)^2 dt/\int_0^\ell f_i^2(t) \eta(t)dt\]
and furthermore estimate 
${\mathcal R}(f_i)$. The above claim then follows from the 
Minmax theorem, applied to the functions $F_1$ and $F_2$
(whose mean may not be zero). Namely, with a small modification of the
initial functions $f_i$ 
we may assure that $F_1(v)=F_2(v)=0$, and we can find a function
in ${\mathcal F}_0(G)$ which vanishes at $v$, with controlled Rayleigh
quotient, as a linear combination of $F_1$ and $F_2$.

To show that ${\mathcal R}(F_i)$ is close to ${\mathcal R}(f_i)$,
by equation (\ref{squareintegral}) it
suffices to compare $\sum_v\sum_{w\sim v}(F_i(w)-F_i(v))^2/p(v)$ to 
$\int_0^\ell (f_i^\prime)^2(t)dt$. 
We carry this estimate out for $f=f_1$ and $F=F_1$, the calculation for
$f_2$ and $F_2$ is identical.

By the definition of an array
of circles, the valency $p(v)$ of every vertex $v$ of the 
base circle equals $2$ or $4$. 
For each $v_j=\alpha(j)$ we have 
\[\sum_{w\sim v_j}(F(w)-F(v_j))^2=
(F(v_{j-1})-F(v_j))^2+(F(v_{j+1})-F(v_j))^2.\] Thus
\begin{align}
 &\sum_v \sum_{w \sim v} (F(v)-F(w))^2 / p(v) \notag \\
\leq  &\frac{1}{2} \sum_{k=1}^{\ell} ( (F(v_{k-1}) - F(v_k))^2 +(F(v_{k+1}) - F(v_k))^2  ) \notag \\  = &
\sum_{k=1}^{\ell} (F(v_{k}) - F(v_{k-1}))^2.\notag \end{align} 

Recall that there is some 
$t_j\in [j-1/2,j+1/2]$ so that $F(v_j)=f(t_j)$. Since $\vert t_{j+1}-t_j\vert \leq 2$, 
the Cauchy Schwarz inequality yields
\begin{equation}
(F(v_{j+1})-F(v_j))^2 =(\int_{t_{j-1}}^{t_j} f^\prime(t)dt)^2
\leq 4 \int_{t_{j-1}}^{t_j} (f^\prime(t))^2 dt.
\end{equation}
Together this implies the estimate
\[\sum_v\sum_{w\sim v}(F(w)-F(v))^2/p(v)\leq 2\int_0^\ell (f^\prime (t))^2dt.\]
As a consequence, we obtain 
${\mathcal R}(F)\leq 4{\mathcal R}(f)$ as desired.

For the estimate of ${\mathcal R}(f)$ (here as before, $f=f_1$)
recall that $0<\beta^\prime(t)\leq 1$ and $f^\prime(t) =
 4\pi \cos(4\pi \beta(t)/E^\prime)\beta^\prime(t)/E^\prime $. Therefore 
\[ (f^\prime(t))^2 = \frac{16\pi^2}{(E^\prime)^2} \cos(4\pi\beta(t)/E^\prime)^2 (\beta^\prime(t))^2 \le 
\frac{16\pi^2}{(E^\prime)^2} \cos(4\pi\beta(t)/E^\prime)^2 \beta^\prime(t). \]

%\[\vert f^\prime(t)\vert =
%\vert 4\pi \cos(4\pi \beta(t)/E^\prime)\beta^\prime(t)/E^\prime\vert \leq
%\vert 4\pi\cos(4\pi \beta(t)/E^\prime)/E^\prime\vert .\]
This implies 
\begin{align}
\int_0^\ell (f^\prime(t))^2dt &\leq \frac{16\pi^2}{(E^\prime)^2} 
\int_0^m \cos(4\pi\beta(t)/E^\prime)^2\beta^\prime(t)dt\notag\\
&= \frac{16\pi^2}{(E^\prime)^2} 
\int_0^{E^\prime/2}\cos(4\pi s/E^\prime)^2ds.\end{align}
With the same argument, using $\eta(t)=E\beta^\prime(t)$, 
we obtain 
\[\int_0^m f^2(t)\eta(t)dt=E\int_0^{E^\prime/2}\sin(4\pi s/E^\prime)^2ds.\]
Since $\int_0^{E^\prime/2}\cos(4\pi s/E^\prime)^2ds=
\int_0^{E^\prime/2}\sin(4\pi s/E^\prime)^2ds$ 
and $E^\prime \ge E^{1/(2^h-2)}$, we deduce 
\[{\mathcal R}(f)\leq 
16\pi^2 E^{-1}(E^\prime)^{-2}\leq 16\pi^2/ E^{2^h/(2^h-2)}
=16\pi^2/{\rm vol}(G)^{2^h/(2^h-1)}.\]
This is what we wanted to show.
\end{proof}

Our next goal is to extend 
Proposition \ref{graphfinished} to generalized arrays of circles.

\begin{proposition}\label{eigenvaluegraphs}
%For every $h>0$ there is a number $q(h)>0$ with the following property.
Let $G$ be a generalized array of circles of depth at most $h$; then
\[\lambda_1(G)\leq 256\pi^2h^{h-1}/3({\rm vol}(G)^{2^h/(2^{h}-1)}).\]
\end{proposition}
\begin{proof}
We prove the proposition by constructing 
for every generalized array of circles $G$ of depth $h$ 
an array of circles $H$ of depth at most $h$, 
and a continuous simplicial surjective map $\Psi:G\to H$.
This construction is done in such a way
that 
\begin{itemize}
\item $\lambda_1(G) \le \frac{4}{3} \lambda_1(H)$,
\item ${\rm vol}(H) \ge h^{-h+1} {\rm vol}(G)$. 
\end{itemize} 

Then from Proposition \ref{graphfinished} we have 
\begin{align}\lambda_1(G) \le \frac{4}{3} \lambda_1(H) &\le \frac{4}{3} 64 \pi^2 / {\rm vol}(H)^{2^h/(2^h-1)} \notag\\
&\le \frac{256 \pi^2}{3} / (h^{-h+1} {\rm vol}(G))^{2^h/(2^h-1)} 
\notag\end{align}
which is what we wanted to show.

For the construction of $H$, note that by the inductive definition, 
a generalized
array of circles of depth at most $h$ differs from an array of circles by
allowing the blow-up of vertices. Recall that this means
that we start with a base
circle $L$, and for each vertex $v$ of $L$, we allow to 
either attach to $v$ a generalized array of circles of depth at most $h-1$, 
or to replace $v$  
by $s\leq h$ arcs $a_1,\dots,a_s$. In the second case, to each such 
arc $a_i$ is associated a positive weight $m_i\geq 1$ so that
$h\geq  \sum_im_i$. To each interior vertex of 
the arc $a_i$ there is
attached a (possibly trivial) generalized array of circles of 
depth at most $m_i-1$, allowing blow-ups of vertices
as before. Let $v_1,v_2$ be the 
common endpoints of the arcs $a_i$. 
Define the \emph{mass} of $a_i$ to be
the total volume of the connected component $E(a_i)$ 
of $G-\{v_1,v_2\}$ 
containing the interior of the arc $a_i$.

The construction of the array of circles $H$ is carried out inductively 
with the following algorithm. 
Begin with the base circle $L$ of $G$. If no vertex of $L$ is blown 
up in $G$ in the inductive build-up of $G$
then repeat the construction with all circles in $G$ 
of depth two. 
Otherwise let $v_1,\dots,v_s$ be the vertices of $L$ which
are blown-up in $G$. For each $i\leq s$, choose a vertex 
arc $a_i$ 
for the vertex $v_i$ with the largest mass. Define
$G_1$ to be the graph obtained from $G$ by collapsing each of the
graphs $E(b_j)$ for all vertex arcs 
$b_j\not=a_i$ for the vertex $v_i$ to a point.
This modification identifies the endpoints of the arc $a_i$.
Or, equivalently, 
in $G_1$, the arc $a_i$ is replaced by a circle of the same length.

Since the sum of the masses of the vertex arcs for the vertex $v_i$ is 
not bigger than $h$ times the mass of $a_i$, 
the volume of $G_1$ is not smaller than ${\rm vol}(G)/h$.
Moreover, $G_1$ 
is a generalized array of circles with no blown-up vertex on the base circle. 
Note that there is a natural 
surjective simplicial map $\Psi_1:G\to G_1$ which 
for each $i$ maps the graph $E(a_i)$ isomorphically, and it maps the 
blown-up base circle in the construction of $G$ to the base circle in $G_1$.

%Figure 3 shows this procedure.
% By construction, $G_1$ is a generalized
%array of circles with base circle $L$, and $L$ does not contain any
%blown-up vertex.

Repeat this construction with $G_1$ and the blown-up circles of depth two.
Since no vertex of the circles of depth $h$ is  blown up, 
in at most $h-1$ such steps we construct in this way an array of circles $H$ with 
${\rm vol}(H)\geq h^{-h+1}{\rm vol}(G)$. There is a natural simplicial
surjection $\Psi:G\to H$. 

Let now $f\in {\mathcal F}_0(H)$ be any function. We show
next that $\int ((\Psi\circ f)^\prime)^2\leq \frac{4}{3}\int (f^\prime)^2$. 
By definition,
$$\int ((f\circ \Psi)^\prime)^2 = 
{\sum_v \frac{1}{p(v)}
\sum_{w \sim v}(f\circ \Psi(v) - f\circ \Psi(w))^2}.$$

Note that 
$\sum_{w \sim v}(f\circ \Psi(v) - f\circ \Psi(w))^2 = 0$ 
if $v$ is an interior vertex of an arc collapsed by $\Psi$. If $v_1, v_2$ are the two endpoints of such an arc, 
and if $v=\Psi(v_1)=\Psi(v_2)$ then
\begin{align}
\sum_{w \sim v_1 \in G}(f\circ \Psi(v_1) - f\circ \Psi(w))^2 &+ \sum_{w \sim v_2 \in G}(f\circ \Psi(v_2) - f\circ \Psi(w))^2\notag \\
& = \sum_{w \sim v \in H}(f(v)- f(w))^2.\notag\end{align}
Also note that $p(v_1) = p(v_2) \ge 3$ while $p(v) = 4$. Hence, 
$$\sum_{i= 1,2} \sum_{w \sim v_i \in G}(f\circ \Psi(v_i) - f\circ \Psi(w))^2 / p(v_i) \le \frac{4}{3} \sum_{w \sim v \in H}(f(v)- f(w))^2 / p(v).$$ 
For all other vertices, both denominator and 
numerator coincide when we switch from $f\circ \Psi$ to $f$. 
Thus we have
\begin{equation}\label{estimate}
 \int_G ((f\circ \Psi)^\prime)^2\leq \frac{4}{3}\int_H(f^\prime)^2\end{equation}
as claimed.

The function $f\circ \Psi$ need not be contained in ${\mathcal F}_0(G)$. Let 
$m=\int f\circ \Psi$ and let 
$\hat f=f\circ \Psi-m$. Then 
$\hat f\in {\mathcal F}_0(G)$, and 
$\int (\hat f^\prime)^2=\int ((f\circ \Psi)^\prime)^2$. Hence using the 
estimate (\ref{estimate}), for the purpose of the proposition it suffices to  
show that $\int \hat f^2\geq \int f^2$.

To this end note that there is a subset of 
the set of vertices of $G$, say the set ${\mathcal V}_1$, which is mapped
by $\Psi$ bijectively onto the set of vertices of $H$: For vertex arcs $a_1,\dots,a_s$ of a blown-up
vertex $v$, with endpoints $v_1,v_2$, choose either $v_1$ or $v_2$ to be in 
${\mathcal V}_1$ and declare the second endpoint as well as all interior vertices
of any erased arc
$a$ and all vertices of
any of the subgraphs of $G$ 
which are attached to interior points of $a$ to be in ${\mathcal V}(G)-{\mathcal V}_1$.
Proceed by induction. 

Now $f\in {\mathcal F}_0(H)$ and therefore $\int_H fm=0$. This implies that
\[\int_G \hat f^2\geq \sum_{v\in {\mathcal V}_1}(f\circ \Psi-m)(v)^2=\int_H (f-m)^2=\int_H(f^2+m^2)\geq \int_H f^2\]
which is what we wanted to show.
\end{proof}

%\begin{figure}[h]
%\begin{center}
%%\includegraphics[scale=0.7]{collapsing-vertex-arcs.pdf}
%\caption{Collapsing vertex arcs $a_1, \ldots, a_{s-1}$ and all attached subgrap%hs}
%\label{fig:collapse}
%\end{center}
%\end{figure}

We are left with finding examples of graphs which realize the bounds
in Proposition \ref{graphfinished} up to a universal constant. 
To this end we say that the
\emph{support ${\rm supp}(f^\prime)$ of the derivative of $f$} consists
of all edges $e$ in $G$ so that the values of $f$ at the endpoints of 
$e$ do not coincide.

We begin with the following elementary

\begin{lemma}\label{separate}
Let $G$ be any finite connected graph.
Assume that there is a decomposition 
${\mathcal F}_0(G)=A\oplus B$ which is orthogonal for the 
$\ell^2$-inner product. Assume furthermore that the supports of the
derivatives of functions in $A,B$ are disjoint; 
then 
\[\lambda_1(G)=\min\{\lambda_1(A),\lambda_1(B)\}\]
where $\lambda_1(A)$ (or $\lambda_1(B)$) 
is the infimum of the Rayleigh quotients
over all functions of the space $A$ (or $B$).
\end{lemma}
\begin{proof}
Under the assumption of the lemma, if $\phi\in {\mathcal F}_0(G)$ 
is arbitrary then $\phi=\alpha +\beta$ for some
$\alpha\in A,\beta\in B$. Since the supports of the 
derivatives of functions in $A,B$ are disjoint, formula (\ref{first})
implies that
\[\int (\phi^\prime)^2=
\int (\alpha^\prime)^2+\int (\beta^\prime)^2.\]

Now if $s=\min\{\lambda_1(A),\lambda_1(B)\}$ then
\[\int (\alpha^\prime)^2\geq s\int \alpha^2, \,
\int (\beta^\prime)^2 \geq s\int \beta^2\] 
and consequently 
\[\int (\phi^\prime)^2=\int (\alpha^\prime)^2+(\beta^\prime)^2
\geq s\int (\alpha^2+\beta^2)=s\int (\alpha+\beta)^2=s\int \phi^2\]
where the second last equality follows from the assumption that 
$\alpha,\beta$ are orthogonal for the $\ell^2$-inner product.
This shows the lemma.
\end{proof}

%The graphs we are interested in will be called 
%\emph{generalized arrays of circles}. 
%First we look at a more restrictive class of graphs
%which is defined as follows.

Recall from Section \ref{thick} the definition of a
step-homogeneous 
array of circles. Such an array $G$ is characterized by
the property that all circles of depth $j$ have the same length
$\ell(j)$. The array of circles is called \emph{optimal} if 
there exists a number $k\geq 3$ so that
$\ell(j)=k^{2^{j-1}}$.

\begin{proposition}\label{cactusone}
For every $h\geq 1$ 
there is a number $q=q(h)>0$ with the following property.
Let $G$ be an optimal  
step-homogeneous array of circles of depth $h$; then 
\[\lambda_1(G)\geq q/{\rm vol}(G)^{2^{h}/(2^h-1)}.\]
%, i.e., $lambda_1(G) {\rm vol}(G)^{(2^h+1)/2^h} \to c$ as $k \to \infty$. 
\end{proposition}
\begin{proof}
Let for the moment $G$ be an arbitrary step homogeneous
array of circles of depth $h$.
Then for every $m\leq h$, 
the union $G_m$ of all circles in $G$ of depth at most $m$ is a
step homogeneous array of circles of depth $m$.
However, if $1\leq m\leq h-2$ then the closure in $G$ of a component
of $G-G_m$ is an array of circles which is not step homogeneous.
Namely, its base circle contains a distinguished vertex (the attaching vertex)
of valency two. Deleting this vertex results in a step homogeneous
array of circles.

%We say that an inequality between numerical invariants 
%of finite connected graphs
%is an \emph{approximate
%equality} for a class of sample graphs if up to a universal multiplicative 
%constant, equality holds for every graph in the class.
Write $A\asymp B$ if $n^{-1}B \leq A\leq nB$ for a universal
constant $n>0$, and write $A\preceq B$ if $A\leq nB$ for a universal 
constant $n>0$. 

Let as before $\ell(j)\geq 3$ be the length of a circle in $G$ of 
depth $j$. 
The volume of $G$  
can recursively be computed by
\[{\rm vol}(G_m)=\ell(m)\chi(m-1){\rm vol}(G_{m-1})\]
where $\chi(m-1)$ is the number of bivalent vertices in $G_{m-1}$
(which is just
$\vert {\mathcal V}(G_{m-1})\vert -\vert {\mathcal V}(G_{m-2})\vert$).
This implies the estimate ${\rm vol}(G)\asymp \prod_{j=1}^h\ell(j)$.
Thus if $G$ is optimal, with base circle of length $k\geq 3$, then 
${\rm vol}(G)\asymp k^{2^{h}-1}$.

 Our goal is 
to show that 
\begin{equation}\label{formula}
\lambda_1(G)\succeq 1/k^{2^h}.\end{equation}

When $h=1$, $G$ is a circle of length $k$, and the
estimate $\lambda_1(G)\geq 4/k^2$ is an easy consequence of the
following. Any function $F$ on the vertex
set of $G$ can be extended
by convex combination to a continuous piecewise
affine function $f$ on all of $G$. If $\sum_vF(v)=0$ then
$\int_G f=0$ where integration is with respect to the
standard Lebesgue measure which gives an edge the volume one.
The Rayleigh quotients can be compared
by ${\mathcal R}(F)\geq \frac{1}{2}{\mathcal R}(f)$. 
Now the smallest non-zero eigenvalue of a smooth
circle of length $R$ equals $4\pi^2/R^2$ which yields the
required estimate for $\lambda_1(G)$.

Furthermore, let 
$f$ be any function on ${\mathcal V}(G)$ 
which either vanishes at a vertex $v$
or changes signs at $v$
(by this we mean that
$f$ assumes a value of opposite sign at a neighbor of $v$).
Cut $G$ open
at $v$, glue two copies of the cut open 
arc to a circle $\hat G$ of double length and extend 
$f$ to a function $F$ on ${\mathcal V}(\hat G)$ by 
reflection at the two copies of $v$ in $\hat G$ (with the 
obvious interpretation if $f$ changes signs at $v$). 
Then $\sum_wF(w)=0$, and the Rayleigh quotients
${\mathcal R}(F)$ and ${\mathcal R}(f)$ can be compared
as follows. 

If $f(v)=0$ then for the two copies $v_1,v_2$ of $v$ in $G$, 
we have 
\[\sum_i\sum_{w\sim v_i}(F(w)-F(v_i))^2=2\sum_{w\sim v}(f(w)-f(v))^2\]
and similarly for the other vertices of $G$ and their two preimages
in $\hat G$, and  consequently ${\mathcal R}(F)=
\mathcal {R}(f)$. 

Now assume that $f$ changes sign at $v$. Let $w_1,w_2$ be the two neighbors of $v$
and assume that the signs of $f(v)$ and $f(w_1)$ are opposite. 
The contribution of the two preimages $v_1,v_2$ of $v$ in $\hat G$ 
in the expression for $\int (F^\prime)^2$ equals
\[(-f(w_1)-f(v))^2+(f(w_1)-f(v))^2+(-f(w_2)-f(v))^2+(f(w_2)-f(v))^2.\]
Now if the signs of $f(w_1)$ and $f(w_2)$ coincide then 
$(-f(w_i)-f(v))^2\leq (f(w_i)-f(v))^2$ for $i=1,2$ 
and hence ${\mathcal R}(F)\leq {\mathcal R}(f)$. Otherwise note
that 
\begin{align}
(-f(w_2)-f(v))^2
&\leq 2((-f(w_2)-f(w_1))^2+(f(w_1)-f(v))^2)\notag\\
&\leq 4(f(w_2)-f(v))^2+6(f(w_1)-f(v))^2\notag\end{align}
which implies that ${\mathcal R}(F)\leq 3{\mathcal R}(f)$.
Thus the Rayleigh
quotient of the function $f$ 
is not smaller than $\frac{1}{3}\lambda_1(\hat G)\geq 
1/3k^2$.

We now proceed by induction on $h$; then case $h=1$ was treated
above. Thus 
assume that the estimate (\ref{formula})
holds true for
optimal step-homoge\-ne\-ous arrays of depth at most $h-1\geq 1$
and 
let $G$ be an optimal step homogeneous array
of depth $h$, with
base circle of length $k$. 

Our strategy is to apply Lemma \ref{separate} to the subspace
of ${\mathcal F}_0(G)$ of 
functions which
are constant on each of the circles of depth $h$
and compare 
their Rayleigh quotients
to $\lambda_1(G_{h-1})$. The following construction 
is used to circumvent the difficulty that a circle of depth $h$ 
in $G$ is
attached to the vertices of $G_{h-1}$ of valence two but not
to every vertex. We construct from $G$ a graph $H$ of uniformly bounded 
valency which is uniformly quasi-isometric to $G$ and which
does not have this problem.
We then use
Theorem 2.1 of \cite{Man05} to compare
$\lambda_1(G)$ to $\lambda_1(H)$.

The graph $H$ is constructed 
successively as follows. If $h\leq 2$ then put $H=G$. Otherwise
for each vertex
$v\in G_{h-2}-G_{h-3}$, collapse one of the two edges in $G_{h-1}-G_{h-2}$ 
which are incident on $v$ to a point. Let $\hat G$ be the  
resulting graph. It arises from a graph of valency four by merging pairs of 
vertices, with any vertex involved in at most one such process.
Thus the valency of $Âž\hat G$ is at most $7$, and 
the collapsing map 
$\hat\Psi:G\to \hat G$ is a one-Lipschitz  
$2$-quasi-isometry which maps
$G_{h-3}$ isomorphically. Note that $\hat G$ 
is obtained from
$\hat\Psi(G_{h-1})$ by attaching to each vertex of 
$\hat\Psi(G_{h-1}-G_{h-3})$ 
a circle of length $k^{2^{h-1}}$.

Repeat this construction with the subgraph $\hat \Psi(G_{h-3})$ 
of $\hat G$, now collapsing edges in $\hat\Psi(G_{h-2}-G_{h-3})$.
In $h-2$ such steps we obtain a graph $H$ and a 
surjective simplicial projection $\Psi:G\to H$ with the
following properties.
\begin{enumerate}
\item The valency of $H$ is at most $4h$.
\item $\Psi$ is an $m$-quasi-isometry 
for a number $m=m(h)\geq 2$ only depending
on $h$ but not on $k$.
\item $Q=\Psi(G_{h-1})$ is $m$-quasi-isometric to 
$G_{h-1}$. 
\item $H$ is obtained from $Q$ by attaching to each vertex $v$ of 
$Q$ a circle $H_v$ of length $k^{2^{h-1}}$.
\end{enumerate}
 
By Theorem 2.1 of \cite{Man05} (note that Mantuano uses the
notion rough isometry for our more standard terminology
quasi-isometry) 
and properties
(1) and (2) above, it now suffices to show the
existence of a number $q=q(h)>0$ so that
$\lambda_1(H)\geq q/k^{2^h}$. By property (3) above, 
by the induction hypothesis 
and by Theorem 2.1 of \cite{Man05}, 
applied to $G_{h-1}$ and its image $Q$ under the map $\Psi$, 
we may 
assume that 
$\lambda_1(Q)\geq q^\prime /k^{2^{h-1}}$ for a 
number $q^\prime>0$ only depending on $h-1$ but not
on $k$.

Let $D\subset {\mathcal F}_0(H)$ 
be the linear subspace of  
all functions on ${\mathcal V}(H)$ 
which are constant on the circles $H_v$
for all vertices $v$ of $Q$. Let
$E\subset {\mathcal F}_0(H)$ be the 
linear subspace of functions which
are constant on $Q$. By definition, 
the supports of the derivatives of 
any two functions $d\in D,e\in E$ are disjoint. 

We claim that ${\mathcal F}_0(H)=D\oplus E$. To this end
let $f\in {\mathcal F}_0(H)$ and 
let $\hat f$ be the unique function on ${\mathcal V}(H)$ 
which coincides with $f$ on $Q\subset H$ 
and is constant on each graph $H_v\subset H$
for every $v\in {\mathcal V}(Q)$. 
Let $a=\sum_{w\in H}\hat f(w)$ and define
\[\Pi(f)=\hat f-a/\vert {\mathcal V}(H)\vert.\]
Then $\Pi:f\in {\mathcal F}_0(H)\to \Pi(f)\in D$ is a linear
projection, i.e. $\Pi$ is linear, maps ${\mathcal F}_0(H)$
into $D$ and equals the identity on the subspace $D$
of ${\mathcal F}_0(H)$. Similarly, 
${\rm Id}-\Pi:{\mathcal F}_0(H)\to E$ is a linear projection
as well.

The subspaces $D,E$ of ${\mathcal F}_0(H)$ 
are not orthogonal for the 
$\ell^2$-inner product, but as 
$\int (\alpha^2+\beta^2)\geq \frac{1}{2}\int (\alpha+\beta)^2$
for any two functions $\alpha,\beta$ on $H$, Lemma \ref{separate}
and its proof implies that 
\[\lambda_1(H)\geq \frac{1}{2}\min\{\lambda_1(D),\lambda_1(E)\}.\]

Our strategy now is to estimate 
$\lambda_1(D)$ and $\lambda_1(E)$ separately. 
We begin with estimating $\lambda_1(D)$.

Thus let $d\in D$ and let 
$d_Q$ be the restriction of $d$ to 
$Q$. By the definition of $D$, we have
${\rm supp}(d^\prime)\subset Q$. The
degree of a vertex 
$v\in Q\subset H$ viewed as a vertex in $H$ 
is at most twice its degree as a vertex in $Q$ 
and therefore  
\begin{equation}\label{rayone}
\int (d_Q^\prime)^2\leq 2\int (d^\prime)^2.\end{equation}

Since for every vertex $v\in Q$ the function 
$d$ is constant on the circle $H_v$ and 
such a circle has precisely $k^{2^{h-1}}$ vertices, we conclude that
\[\sum_{w\in H_v}d^2(w)=k^{2^{h-1}}d_Q^2(v)\]
and hence
\begin{equation}\label{raytwo}
k^{2^{h-1}}\int d_Q^2= \int d^2.\end{equation}
The estimates (\ref{rayone},\ref{raytwo}) imply that
\[{\mathcal R}(d)\geq {\mathcal R}(d_Q)/2k^{2^{h-1}}.\]

On the other hand, we also have
\[\int d=k^{2^{h-1}}\int d_Q\]
and therefore 
$d_Q\in {\mathcal F}_0(Q)$ and hence
${\mathcal R}(d_Q)\geq \lambda_1(Q)$. 
Thus by the induction 
hypothesis, we obtain 
\begin{equation}\label{raythree}\lambda_1(D)
\geq q^\prime/2k^{2^{h-1}}k^{2^{h-1}}=q/k^{2^h}\end{equation}
where $q=q^\prime/2$ 
only depends on $h$.

We are left with estimating $\lambda_1(E)$. 
To this end define another graph $W$ as follows.
The graph $W$ contains a distinguished vertex $w$. 
There are $n=\vert {\mathcal V}(Q)\vert$ edges incident on $w$.
Let $e$ be such an edge; one endpoint of $e$ equals $w$.
Attached to the second endpoint is 
a circle with $k^{2^{h-1}}$ vertices.

Note that $W$ admits a group of automorphisms
which fix $w$ and permute the edges of $W$ incident on $w$.
Each permutation of the edges incident on $w$ is the
restriction of such an automorphism.
Any labeling of the edges of $W$ incident on 
$w$ gives rise to a 
bijection ${\mathcal V}(H)\to 
{\mathcal V}(W)-\{w\}$ which maps the vertices in $Q$ to 
endpoints of the edges incident on $w$.
Fix once and for all such a bijection
$\Theta$.

Via the map $\Theta$, 
each function $f\in E$ naturally induces 
a function
$f^*$ on ${\mathcal V}(W)$. This function 
may not be of zero mean, but as $f$ is of zero mean
and the map $\Theta$ is a bijection 
of set of the vertices of $H$ onto the set of 
vertices of $W$ distinct from $w$, 
the square norm of the normalization 
$g$ of $f^*$ is not
smaller than $\sum_{v\in H}f^2(v)$ (compare the 
proof of Proposition \ref{eigenvaluegraphs}).

As $f\in E$, as $\Theta$ maps vertices 
of degree contained in $[2,4h]$ to 
vertices of degree in $[2,3]$, and as the special 
vertex $w$ does not contribute to $\int (g^\prime)^2$,   
we have $\int (g^\prime)^2\leq 4h\int (f^\prime)^2$.
Together this shows 
\[{\mathcal R}(f)\geq \frac{1}{4h}{\mathcal R}(g).\]
As a consequence, for the desired estimate
of $\lambda_1(E)$ it suffices to show that
$\lambda_1(W)\geq m/k^{2^h}$ for a universal constant
$m>0$. 

To this end
let $f$ be an eigenfunction 
on $W$ for the smallest eigenvalue $\lambda_1(W)$. 
If we define
\begin{equation}\label{Laplacian}
{\mathcal L}f(u)=\frac{1}{\sqrt{p(v)}}
\sum_{w\sim v}(\frac{f(v)}{\sqrt{p(v)}}-
\frac{f(w)}{\sqrt{p(w)}})\end{equation}
then ${\mathcal L}f(u)=\lambda_1(W)f(u)$.

%The restriction of $f$ to a component of 
%$W-{w}$ is an eigenfunction on a circle of length
%$k^{2^{h-1}}$, with an edge attached at a single
%vertex $u$. 

We distinguish now two cases. In the first case,
$f(w)=0$. Then equation (\ref{Laplacian}) shows that
the restriction $f_U$ of $f$ to the closure of 
each component $U$ of
$W-\{w\}$ is an eigenfunction on $U$ for the eigenvalue
$\lambda_1(W)$. Such a component $U$ is a circle
of length $k^{2^{h-1}}$ with a single edge attached at
one vertex. 
As $f_U$ assumes the value zero,
its Rayleigh quotient ${\mathcal R}(f\vert U)$ can be
estimated by 
\[{\mathcal R}(f\vert U)\succeq (k^{2^{h-1}})^2=k^{2^h}\]
by the discussion in the
beginning of this proof which is equally valid for a circle
with a single edge attached at one vertex instead of a circle.

Together this yields
\begin{align}\int (f^\prime)^2\geq \sum_U 
\sum_{v\in {\mathcal V}(U)-\{w\}}\frac{1}{p(v)}\sum_{z\sim v}(f(z)-f(v))^2\notag\\
\succeq  (\sum_U\sum_{v\in {\mathcal V}(U)-\{w\}}f^2(v))/k^{2^h}\notag\end{align}
and therefore ${\mathcal R}(f)0\lambda_1(W)\succeq 1/k^{2^h}$
as desired.

Now assume that $f(w)\not =0$. 
Let $A$ be an automorphism of $W$; then 
$f\circ A$ is an eigenfunction for the eigenvalue $\lambda_1(W)$.
If $f\circ A\not=f$, then $f-f\circ A$ is an eigenfunction on $W$ for the
eigenvalue $\lambda_1(W)$ which vanishes at $w$. 
The desired estimate now follows from the above discussion
provided that $f\circ A\not=A$ for at least one automorphism
$A$ of $W$. 

Finally suppose that $f\circ A=f$ for all 
automorphisms $A$ of $W$. Then $f$ descends to 
an eigenfunction $\hat f$  on a circle $U$ 
of length $k^{2^{h-1}}$ with a single
edge attached at one vertex, and of zero mean. Here the value 
of $\hat f$ at the unique vertex of $U$ of degree one equals 
$f(w)/\vert {\mathcal V}(Q)\vert$. The 
eigenvalue of $\hat f$ equals $\lambda_1(W)$.
Then $\lambda_1(W)\geq \lambda_1(U)$ and hence 
as before, we conclude that
$\lambda_1(W)\succeq 1/k^{2^h}$ for a universal constant $q$.
Together this shows that indeed $\lambda_1(E)\succeq 1/k^{2^h}$ as claimed.

Since we have established that 
$\lambda_1(D)\succeq 1/k^{2^h}$ and $\lambda_1(E)\succeq 1/k^{2^h}$, we get 
$\lambda_1(G)\asymp \lambda_1(H)\succeq 1/k^{2^h}$. 
This shows the proposition.
\end{proof}

\begin{remark}  The constant $q(h)$ in Proposition \ref{cactusone} can be made 
effective. However, this would require a considerable effort in bookkeeping.
Moreover, our proof would yield an exponential decay of $q(h)$ in $h$.
\end{remark}

%%%%%%%%%%%%%%%%%%%%%%%%%%%%%%%%%%%%%%%%%%%%%%%%%%%%%%%%%%%%%%%%%%%%%%%%%%%
%%%%%%%%%%%%%%%%%%%%%%%%%%%%%%%%%%%%%%%%%%%%%%%%%%%%%%%%%%%%%%%%%%%%%%%%%%%

%%%%%%%%%%%%%%%%%%%%%%%%%%%%%%%%%%%%%%%%%%%%%%%%%%%%%%%%%%%%%%%%%%%%%%%%%%%
%%%%%%%%%%%%%%%%%%%%%%%%%%%%%%%%%%%%%%%%%%%%%%%%%%%%%%%%%%%%%%%%%%%%%%%%%%%
\section{The smallest eigenvalue of mapping tori}\label{smallest}

In this section we use the results from Section \ref{arrays}
and Section \ref{thick} to prove Theorem \ref{thm1} from the 
introduction.

As explained in Section \ref{thick}, 
a hyperbolic 3-manifold $M$ can be decomposed
as $M=M_{\rm thick}\cup M_{\rm thin}$ where $M_{\rm thick},M_{\rm thin}$ are 
smooth manifolds with boundary $\partial M_{\rm thick},
\partial M_{\rm thin}$.
Each component of $M_{\rm thin}$ is 
a \emph{Margulis tube}. Such a tube $T$ is a tubular neighborhood
of a geodesic $\gamma$ in $M$ of length less than $2\epsilon$
where $\epsilon >0$ is a \emph{Margulis constant}
for hyperbolic 3-manifolds. 
The geodesic $\gamma$ is called 
the \emph{core geodesic} of the tube.

The thick part $M_{\rm thick}$ of $M$ is a 
smooth submanifold of $M$ with boundary. 
Thus the spectrum of $M_{\rm thick}$ with Neumann boundary conditions is defined.
This spectrum is discrete, with finite multiplicities. Constant functions are eigenfunctions
corresponding to the smallest eigenvalue $\lambda_0=0$. 
Let $\lambda_1(M_{\rm thick})$ be the smallest non-zero eigenvalue with Neumann
boundary conditions. We now evoke the main result of \cite{Man05}  to show

\begin{proposition}\label{firstbound}
For every $g\geq 2$ there is a number $c_3=c_3(g)>0$ with the following property.
 Let $M$ be a hyperbolic mapping torus of genus $g$; then 
\[\lambda_1(M_{\rm thick})\leq c_3/{\rm vol}(M)^{2^{2g-2}/(2^{2g-2}-1)}.\]
\end{proposition}
\begin{proof}
By Proposition \ref{model}, $M_{\rm thick}$ is $L$-quasi-isometric 
to a generalized
array $G$ of circles of depth at most $2g-2$ for a number $L=L(g)>0$ only
depending on $g$.
The main result of \cite{Man05} applies
to the Laplacian on manifolds with boundary and Neumann boundary condition
(see \cite{LS12} for a more precise statement along these lines) and shows that
there is a number $b>0$ only depending on $g$  
so that $\lambda_1(M_{\rm thick})\leq b \lambda_1(G)$. 

Proposition \ref{eigenvaluegraphs} yields that 
\[\lambda_1(G)\leq 256\pi^2(2g-2)^{2g-3}/{\rm vol}(G)^{2^{2g-2}/(2^{2g-2}-1)}.\]
 The proposition now follows from the fact that
 ${\rm vol}(G)\sim {\rm vol}(M_{\rm thick})$ (by uniform quasi-isometry) and
 ${\rm vol}(M_{\rm thick})\geq \frac{2}{3} {\rm vol}(M)$ for 
 a suitable choice of a Margulis constant 
(by the explicit description of the metric in a Margulis tube, 
see \cite{H16} for a more comprehensive
 discussion).
\end{proof}

The following is shown in \cite{H16}.

\begin{proposition}\label{compare}
There exists a number $d>0$ and a suitable choice of a 
Margulis constant such that
\[\frac{1}{48}\lambda_1(M_{\rm thick})\leq
\lambda_1(M)\leq d\log {\rm vol}(M_{\rm thin})\lambda_1(M_{\rm thick})\] 
for every hyperbolic 3-manifold $M$. 
\end{proposition}

We are now ready to show Theorem \ref{thm1} from the 
introduction.

\begin{corollary}\label{precise}
For every $g\geq 2$ there is a number $C_1=C_1(g)>0$
with the following properties.
\begin{enumerate}
\item 
\[\lambda_1(M)\leq C_1\log {\rm vol}(M)/{\rm vol}(M)^{2^{2g-2}/(2^{2g-2}-1)}.\]
for every hyperbolic mapping torus $M$ of genus $g$. 
\item There exists a sequence $M_i$ of hyperbolic mapping tori of genus $g$
with ${\rm vol}(M_i)\to \infty$ and
such that 
\[\lambda_1(M_i)\geq C_1^{-1} /{\rm vol}(M_i)^{2^{2g-2}/(2^{2g-2}-1)}.\]
\end{enumerate}
\end{corollary}
\begin{proof} The first part of the corollary is immediate 
from Proposition \ref{firstbound} 
and from 
Proposition \ref{compare}.  

To show the second part, 
let $g\geq 2$ be arbitrary. By Proposition \ref{example}, there
exists a sequence $M_i$ of mapping
tori of genus $g$ so that ${\rm vol}(M_i)\to \infty$ and that for each $i$, 
$(M_i)_{\rm thick}$ is uniformly quasi-isometric to an optimal 
step homogeneous array of circles $G_i$ of depth $2g-2$.

By Proposition \ref{cactusone}, the first eigenvalue of the array $G_i$ is 
not smaller than $\hat q/{\rm vol}(G_i)^{2^{2g-2}/2^{2g-2}-1}$ where $\hat q = \hat q(2g-2)$ is a universal 
constant.

Using once more the main result of \cite{Man05} 
and the volume comparison 
\[{\rm vol}(G_i)\asymp {\rm vol}((M_i)_{\rm thick})\geq 
\frac{2}{3}{\rm vol}(M_i)\] 
for a suitable choice of a Margulis constant, 
we conclude that there is 
a universal constant $c^\prime>0$ so that 
\[\lambda_1((M_i)_{\rm thick})\geq c^\prime/{\rm vol}(M_i)^{2^{2g-2}/(2^{2g-2}-1)}.\]
The second part of the
corollary now follows from the first inequality of 
Proposition \ref{compare}.
\end{proof}

 We complete this section 
by estimating the smallest eigenvalue of mapping tori
 defined by periodic Teichm\"uller geodesics in moduli space
${\mathcal M}(S)$ 
which spend a definite proportion of time in the
$\epsilon$-thick part ${\mathcal M}(S)_\epsilon$ 
of surfaces with injectivity radius
bigger than $\epsilon$. Note that ${\mathcal M}(S)_\epsilon$
is the quotient of the $\epsilon$-thick part of Teichm\"uller space
under the action of the mapping class group.

 \begin{proposition}\label{thickpart}
 For sufficiently small $\epsilon >0$ 
there exists a number $b=b(g,\epsilon)>0$ with the following 
property. Let 
$S_R^1$ be the circle of length $R>0$ and 
let $\gamma:S^1_R\to {\mathcal M}(S)$ 
be a periodic Teichm\"uller geodesic of length $R$. 
Let $p=3c_2(g,\epsilon)>0$ be as in Proposition \ref{control}
and 
let $Q=\{t\in S^1_R\mid
 \gamma(t-p,t+p)\subset {\mathcal M}(S)_\epsilon\}$; if the 
Lebesgue measure of $Q$ 
 is at least $\zeta R$ for some $\zeta\in (0,1)$ then 
 \[\lambda_1(M)\leq \frac{b}{\zeta^2{\rm vol}(M)^2}.\]
 \end{proposition}
 \begin{proof} Let $c_2=c_2(g,\epsilon)>0$ be as in 
Proposition \ref{control}, let $p = 3 c_2$ and let $\zeta\in (0,1)$. 
Let $\gamma:S^1_R\to 
 {\mathcal M}(S)$ be a periodic Teichm\"uller 
geodesic as in the proposition for this number $\zeta$.
By continuity, the set $Q\subset S_R^1$ is open and hence 
it is a union of at most countably many open
intervals. By the definition of $Q$, the $p$-neighborhoods of these
intervals are pairwise disjoint. 
By assumption, the Lebesgue measure of $Q$ is at least $\zeta R$. 
Choose finitely many connected components $I_1,\dots,I_s\subset Q$ of
Lebesgue measure at least $\zeta R/2$. Assume that these intervals are linearly
ordered along $[0,R]$. 
Let $u_j$ be the length of $I_j$.

By Proposition \ref{control}, for each $j\leq s$ there exists a submanifold
$N_j$ of $M$ with smooth boundary which is diffeomorphic to 
$S\times [0,1]$. This diffeomorphism is chosen to be
compatible with the orientation of $S_R^1$ defined by 
the parametrization of $\gamma$. Thus the two boundary 
components of $N_j$ are naturally ordered. We denote by
$\partial N_j^-$ the component which is smaller for this order,
and by $\partial N_j^+$ the bigger component.
The submanifolds $N_j$ are pairwise
disjoint, and $M-\cup_i N_i$ consists of $s$ connected components
$P_1,\dots,P_s$ diffeomorphic to 
$S\times [0,1]$. 
We have $\partial P_i=\partial N_i^+\cup 
\partial N_{i+1}^-$.

For each $j$ 
there exists a smooth surjective map 
\[f_j:N_j\to [\sum_{i<j}u_i, \sum_{i\leq j}u_j]\]
of uniformly bounded derivative. 
Write $u=\sum_iu_i\geq \zeta R/2$ and define a function
$f:M\to [0,u]$ by $f\vert N_i=f_i$ and by the requirement that $f$ is constant on each of the
manifolds $P_j$. We can modify $f$ so that its derivative is uniformly bounded.
Define functions $\alpha, \beta$ on $[0,u]$ as 
\begin{equation}
\alpha(s)=
\begin{cases}
\sin(\pi s/ u), &\text{if $0\leq s\leq u/2$;}\\
0 &\text{if $u/2\leq s\leq u$.}
\end{cases}
\notag \end{equation}
and 
\begin{equation}
\beta(s)=
\begin{cases}
0, &\text{if $0\leq s\leq u/2$;}\\
\sin(\pi (s-u/2)/ u) &\text{if $u/2\leq s\leq u$.}
\end{cases}
\notag \end{equation}
Then $\alpha \circ f, \beta \circ f$ are smooth, with supports intersecting
in a zero volume set,   
and their Rayleigh quotients 
are uniformly equivalent to $1/u^2$. To this end note that the
Rayleigh quotients of 
$\alpha, \beta$ are 
$\pi^2/u^2$, and since $f$ has uniformly bounded derivative, 
the Rayleigh quotients of $\alpha, \beta$ are uniformly equivalent 
to the Rayleigh quotients of 
$\alpha \circ f, \beta \circ f$.

By the Minmax-theorem for the spectrum of the Laplacian, we know
that for any set of functions $\rho_0,\dots,\rho_k:M\to \mathbb{R}$
whose supports pairwise intersect on zero-volume sets,
we have $\lambda_k\leq\max\{{\mathcal R}(\rho_i)\mid
0\leq i\leq k\}$ (compare \cite{White}) and therefore
$\lambda_1(M) \leq \max\{{\mathcal R}(\alpha \circ f), {\mathcal R}(\beta \circ f)\}$. 
As a result, $\lambda_1(M) \leq d/u^2$ where $d>0$ is a constant only depending on $g, \epsilon$.

We are left with showing that for fixed $\zeta$, the volume of $M$ is uniformly 
equivalent to 
$u$. To this end we evoke from \cite{B03,BB14,KMcS14}
that the volume of $M$ is equivalent to the translation
length  for the Weil-Petersson metric
of the pseudo-Anosov element defining $\gamma$,
and this translation length is bounded from above by the 
length of $\gamma$ for the Teichm\"uller metric up to a factor 
which only depends on $g$
(see e.g. \cite{KMcS14}). Thus the volume 
of $M$ is bounded from above by $\chi R$ where $\chi>0$ only depends on 
$\zeta,g,\epsilon$ and is linear in $\zeta$. 
Since $u \geq \zeta R/2$, the proposition follows.  
\end{proof}

%%%%%%%%%%%%%%%%%%%%%%%%%%%%%%%%%%%%%%%%%%%%%%%%%%%%%%%%%%%%%%%%%%%%%%%%%%%
%%%%%%%%%%%%%%%%%%%%%%%%%%%%%%%%%%%%%%%%%%%%%%%%%%%%%%%%%%%%%%%%%%%%%%%%%%%
\section{Typical mapping tori} \label{typical}

The goal of this section is to show that a typical mapping torus satisfies the hypothesis
in Proposition \ref{thickpart} for a number $\zeta>0$ only depending on the genus for some fixed
sufficiently small number $\epsilon>0$. We then evoke Proposition \ref{thickpart} to conclude
the proof of Theorem \ref{thm2} for typical mapping tori.

Let $\mathcal{G}(L)$ be the set of all conjugacy classes of 
pseudo-Anosov mapping classes (in short: p-A mapping classes) 
on $S$ whose translation length on Teichm\"uller space 
${\mathcal T}(S)$ is less than $L$. 
It is known that $\mathcal{G}(L)$ is finite for any fixed $L > 0$. 
Up to isometry, a hyperbolic mapping torus only depends on the
conjugacy class of the defining pseudo-Anosov element. 
We say a typical mapping torus (or a typical p-A conjugacy class) 
satisfies a property $(\ast)$ if 
$$ \dfrac{|\{ \phi \in \mathcal{G}(L) : \phi \mbox{ satisfies property } (\ast) \}|}{|\mathcal{G}(L)|} \to 1 $$ as $L 
\to \infty$. 
In this section we prove the following.
\begin{proposition}
\label{prop:thickpartportion-typical}
Let $U\subset {\mathcal T}(S)$ be an open ${\rm Mod}(S)$-invariant
set which contains the axis of at least one pseudo-Anosov element. 
For each $p > 0$, there exists 
$\delta=\delta(U,p)>0$ with the following property. 
The proportion of time along an axis $\gamma$ of 
a typical pseudo-Anosov element consisting of
points $\gamma(t)$ so that the segment $\gamma[t-p,t+p]$ is entirely contained
in $U$ is at least $\delta$.
\end{proposition}
To prove this proposition, we will use the 
equidistribution of closed orbits of the Teichm\"uller geodesic flow in the space of unit area quadratic differentials, 
obtained in \cite{Hamenstaedt13}. 
Let $\widetilde{Q^{1}(S)}\to {\mathcal T}(S)$ 
be the bundle of unit area quadratic differentials, 
which can be identified with the unit cotangent bundle to $\T(S)$, 
and let $Q^{1}(S)=\widetilde{Q^{1}(S)}/\Mod(S)$. 

The \emph{Teichm\"uller flow}Ê
$\Phi^{t}:Q^{1}(S)\to Q^{1}(S)$ acts on $Q^1(S)$ preserving a
finite measure $\lambda$ in the Lebesgue measure class, the 
\emph{Masur-Veech measure}. The measure $\lambda$ has full support.

We can identify conjugacy classes of pseudo-Anosov elements with closed orbits of the 
Teichm\"uller geodesic flow.

For a closed orbit $\gamma$ let $l(\gamma)$ denote its length. 

Let $\delta_\gamma$ be the standard flow-invariant Lebesgue measure
on $\gamma$ of total mass $l(\gamma)$.

For a Borel subset $A$ of ${\mathcal Q}^1(S)$ let 
$l(A)=\delta_{\gamma}(A)$. 

%This is independent of the
%choice of a parametrization of $\gamma$.
For each $L>0$ we may define a measure on ${\mathcal Q}^1(S)$ by $$\lambda_{L}=\frac{1}{L|\mathcal{G}(L)|}\sum_{\gamma\in \mathcal{G}(L)}\delta_{\gamma}.$$
The main result of \cite{Hamenstaedt13} shows:

\begin{lemma} \label{equid}
The measures  
$$\lambda_{L}=\frac{1}{L|\mathcal{G}(L)|}\sum_{\gamma\in \mathcal{G}(L)}\delta_{\gamma}$$ weakly converge to $\lambda$ as $L\to \infty$.
\end{lemma}
%The main result of \cite{Hamenstaedt13} shows that the measures defined by 
%$$\lambda_{L}=\frac{1}{L|\mathcal{G}(L)|}\sum_{\gamma\in \mathcal{G}(L)}\delta_{\gamma}$$ weakly converge to $\lambda$.
By the classical Portmanteau theorem Lemma \ref{equid} can be rephrased as follows.
\begin{lemma}\label{equidopen}
For any Borel set $V\subset Q^{1}(S)$ whose boundary has measure zero we have 

$$\lim_{L\to \infty} \frac{1}{L\vert {\mathcal G}(L)\vert}
\sum_{\gamma\in \mathcal{G}(L)}l(\gamma\cap V)=\lambda(V)$$

%$$\lambda(V) \leq \lim \inf_{L\to \infty} \frac{1}{L\vert {\mathcal G}(L)\vert} \sum_{\gamma\in \mathcal{G}(L)}l(\gamma\cap V)$$ and 
%$$\lim \sup_{L\to \infty} \frac{1}{L\vert {\mathcal G}(L)\vert}\sum_{\gamma\in \mathcal{G}(L)}l(\gamma\cap V)\leq \lambda(\overline{V})$$
\end{lemma}

We will use Lemma \ref{equid} together with the ergodicity of the Teichm\"uller geodesic flow to prove the following.

\begin{proposition}\label{typicalequid}
Let $U\subset Q^{1}(S)$ be a Borel subset whose boundary has Lebesgue measure zero.
Then for any $\epsilon>0$ a typical Teichm\"uller geodesic spends a proportion between $(1-\epsilon)\mu(U)$ and $(1+\epsilon)\mu(U)$ in $U$.
\end{proposition}

\begin{proof}
 It suffices to prove that for each $\epsilon>0$, a typical closed orbit spends a proportion at least $(1-\epsilon)\mu(U)$ in $U$ (the upper bound can then be obtained by replacing $U$ with its complement).
 Fix $\epsilon>0$.
For each $L>0$ let $A(L)\subset \mathcal{G}(L)$ denote the set corresponding to closed orbits of length at most $L$
that spend a proportion at most $(1-\epsilon)\lambda(U)$ in $U$. Define for each $L>0$ a finite measure $$\kappa_{L}=\frac{1}{L|\mathcal{G}(L)|}\sum_{\gamma\in \mathcal{G}(L)}\delta_{\gamma}.$$
To prove Proposition \ref{typicalequid} it suffices to prove that the measures $\kappa_{L}$ weakly converge to zero.
Note, $\kappa_{L}\leq \lambda_{L}$ so by Lemma \ref{equid} any subsequence of the $\kappa_{L}$ has an weak accumulation point, which is a finite measure on $Q^{1}(S)$.

Let $\kappa$ be the weak limit of $\kappa_{L_i}$ for some sequence $L_i$. 
As $\kappa_{L}(U)\leq (1-\epsilon)\kappa_{L}({Q^{1}(S)})$ for each $L$ and the boundary of $U$ has measure $0$ we have $$\kappa(U)\leq (1-\epsilon)\kappa ({Q^{1}(S)}).$$
Since the $\kappa_L$ are $\Phi^t$ invariant so is $\kappa$. Thus, by ergodicity of the Teichm\"uller flow with respect to $\lambda$ we have $\kappa=c\lambda$ for some $c=\kappa(Q^{1}(S))\geq 0$. 
Then $\kappa(U)=\lambda(U)\kappa ({Q^{1}(S)})$, providing a contradiction and completing the proof.
\end{proof}

We now conclude the proof of
Proposition \ref{prop:thickpartportion-typical}.
\begin{proof}[Proof of Proposition \ref{prop:thickpartportion-typical}]
Let $U\subset {\mathcal T}(S)$ be a nonempty open set containing an axis of a pseudo-Anosov element.
Let $W$ be the preimage of $U$ in $\widetilde{Q^{1}(S)}$ and $V$ the image of $W$ in $Q^{1}(S)$.
For $T>0$, let $V_{T}\subset {Q^{1}(S)}$ be the subset of $q$ such that $\Phi^{t}q\in V$ for all $t\in (-T,T).$ 
By finiteness of $\lambda$ we know that $\lambda(\partial V_{T})=0$ for all but countable many $T$, so for given $p>0$ we may find a $T>p$ with this property. Then by Proposition \ref{typicalequid} for any $\epsilon>0$ a typical closed orbit of the Teichm\"uller flow spends a proportion at least $(1-\epsilon)\mu(V_T)$ in $V$. Thus, the proportion of time along an axis $\gamma$ of 
a typical pseudo-Anosov element consisting of
points $\gamma(t)$ so that the segment $\gamma[t-p,t+p]$ is entirely contained
in $U$ is at least $\mu(V_T)$. 
Since $U$ is open so is $V_T$, and since $U$ contains an axis of a pseudo-Anosov element we know $U_T$ is nonempty so $\mu(V_T)>0$. This completes the proof with 
$\delta(p)=\lambda(V_T)/2>0$.
\end{proof} 

%Let $U\subset {\mathcal T}(S)$ be a set as in the proposition let $\tilde C\subset \widetilde{Q^1(S)}$ be the preimage of $U$.
%By invariance of $U$ under the action of ${\rm Mod}(S)$, the set $\tilde C$ projects to an open subset of $Q^1(S)$ which
%contains a periodic orbit $\gamma$ for the Teichm\"uller flow.

As a result, we obtain Theorem \ref{thm2} for typical mapping tori.

\begin{corollary}
For every $g\geq 2$ there exists a number $\kappa=\kappa(g)>0$ so that 
\[\lambda_1(M)\in [\frac{1}{\kappa(g){\rm vol}(M)^2},\frac{\kappa(g)}{{\rm vol}(M)^2}]\]
for a typical mapping torus of genus $g$.
\end{corollary}
\begin{proof}
Choose $\epsilon >0$ sufficiently small that the open 
${\rm Mod}(S)$-invariant subset
${\mathcal T}(S)_\epsilon\subset {\mathcal T}(S)$ of all surfaces whose
systole is bigger $\epsilon$ contains the axis of a pseudo-Anosov element.
Let $p=3c_2(g,\epsilon)>0$ as in Proposition \ref{control}. 
Proposition \ref{prop:thickpartportion-typical} shows 
there exists a number $\zeta>0$ such that for a typical periodic 
Teichm\"uller geodesic 
$\gamma$, the proportion of time $t$ so that the
segment $\gamma(t-p,t+p)$ is entirely contained in ${\mathcal M}(S)_\epsilon$
is at least $\zeta$. The corollary now follows from Proposition \ref{thickpart}. 
\end{proof}

%%%%%%%%%%%%%%%%%%%%%%%%%%%%%%%%%%%%%%%%%%%%%%%%%%%%%%%%%%%%%%%%%%%%%%
%%%%%%%%%%%%%%%%%%%%%%%%%%%%%%%%%%%%%%%%%%%%%%%%%%%%%%%%%%%%%%%%%%%%%%%%%%%

\section{Random mapping tori}\label{random}

The main goal of this section is to prove Theorem \ref{thm3} from the
introduction. The part of Theorem \ref{thm2} concerning 
radom mapping tori follows from this and Proposition \ref{thickpart}.

We begin with reviewing some background on random walks on groups. This is a vast
subject,
see for example   \cite{Kaimanovich-Masur} \cite{Horbez-Dahmani}
\cite{Maher-Tiozzo} for more details.

Let $G$ be a countable finitely generated group.
Let $\mu$ be a symmetric probability measure on $G$ and
let $\mu^{\mathbb{Z}}$ be the product measure on $G^{\mathbb{Z}}$.

Let $T:G^{\mathbb{Z}}\to G^{\mathbb{Z}}$ be the following invertible
transformation:
$T$ takes the two-sided sequence $(h_{i})_{i\in \mathbb{Z}}$
to the sequence $(g_{i})_{i\in \mathbb{Z}}$ with
$g_{0}=e$ and $g_{n}=g_{n-1}h_n$ for $n\neq 0$.
Explicitly, this means
$$g_{n}=h_{1}\cdots h_{n}\quad \text{ for }n>0$$ and
$$g_{n}=h^{-1}_{0}h^{-1}_{-1}\cdots h^{-1}_{-n+1}\quad \text{ for }n<0.$$

Similarly, let $\mu^{\mathbb{N}}$ be the product measure on $G^{\mathbb{N}}$.
Let $T_{+}:G^{\mathbb{N}}\to G^{\mathbb{N}}$ be the transformation that
takes the one-sided infinite sequence $(h_{i})_{i\in \mathbb{N}}$
to the sequence $(g_{i})_{i\in \mathbb{N}}$ with
$g_{0}=e$ and $g_{n}=g_{n-1}h_n$ for $n\neq 0$.
Explicitly, for $n>0$ this means
$$g_{n}=h_{1}\cdots  h_{n}.$$

Let $\overline{P}$ be the pushforward measure
$T^{*}\mu^{\mathbb{Z}}$ and $P$ the pushforward measure
$T^{*}_{+}\mu^{\mathbb{N}}$.

The measure $P$ describes the distribution $\mu$ on
sample paths, i.e. of products of independent $\mu$-distributed increments.
        The measure space $(G^{\mathbb{Z}},\overline{P})$ is
naturally isomorphic to $(G^{\mathbb{N}},P)\otimes (G^{\mathbb{N}},P)$
via the map sending the bilateral path $\omega$ to the pair of unilateral
paths
$((\omega_{n})_{n\in \mathbb{N}},(\omega_{-n})_{n\in \mathbb{N}})$.

Assume now that $G$ acts by isometries on a metric space $(X, d_X)$
and let $x_0 \in X$.
If $\mu$ has finite first moment (which is obviously the case
if the support of $\mu$ is finite, which is the case of interest for us),
Kingman's subadditive ergodic theorem
implies that for $P$ a.e. sample path $\omega$
the limit
$$L_X=\lim_{n\to \infty}\frac{d_X(\omega_{n}x_{0},x_{0})}{n}$$ exists.
This number $L$ is called the
\emph{drift} of the random walk
induced by $\mu$ with respect to the metric $d_X$.

In the case that $(X,d_X)$ is a separable Gromov hyperbolic metric space
then
the action of $G$ on $X$ is called \emph{nonelementary} if $G$ contains a
pair of loxodromic
isometries with disjoint sets of fixed points in
the \emph{Gromov boundary} $\partial X$ of $X$. A symmetric probability
measure on $G$ is called \emph{nonelementary}
if the subgroup of G generated by its
support is a nonelementary subgroup of $G$.
The following results are due in this generality to Maher and Tiozzo
\cite{Maher-Tiozzo}.

\begin{theorem}\label{Maher-Tiozzo}
 Let $G$ be a countable group that acts by isometries on a
separable Gromov hyperbolic space $(X,d_X)$ such that any two points in $X\cup \partial X$ can be connected by a geodesic. 
Let $\mu$ be a nonelementary probability measure
on $G$. 
\begin{enumerate}
\item For any $x\in X$ and $P$ a.e.
sample path $\omega=(\omega_{n})_{n\in \mathbb{N}}$ of the random walk
on $(G, \mu)$, the sequence $(\omega_{n}x_0)_{n\in \mathbb{N}}$ converges
to a point ${\rm bnd}(\omega)\in \partial X$.
\item 
If $\mu$ has finite first moment with respect to the metric
$d_X$, then there exists $L_{X}>0$ such that
for $P$-a.e. sample path $\omega$ one has
$$\lim_{n\to \infty}\frac{d_{X}(x_{0},\omega_{n}x_{0})}{n}=L_{X}.$$
Moreover,
there is a unit speed geodesic ray $\tau$ converging to ${\rm bnd}(\omega)$
such that
$$\lim_{n\to \infty}\frac{d_{X}(\tau(L_Xn),\omega_{n}x_{0})}{n}=0.$$
\item 
If $\mu$ has finite support, then for $P$-a.e. sample path
$\omega$ there is an $n_{0}$ such that $\omega_{n}$ acts loxodromically for
all $n\geq n_{0}$.
\end{enumerate}
\end{theorem}

Using techniques of \cite{Maher-Tiozzo}, Dahmani and Horbez
(Proposition 1.9 of \cite{Horbez-Dahmani})
proved. 
\begin{proposition}
\label{axesclosetobasepoint}
Let $X$ be a separable geodesic Gromov hyperbolic metric space, with
hyperbolicity constant $\delta$. For all $A>0$, there exists
a number $\kappa=\kappa(A, \delta)>0$ such that the following
holds. Let $G$ be a group acting by isometries on $X$, and let $\mu$ be a
nonelementary
probability measure on $G$ with finite support. Let $L_{X}>0$ denote the
drift of the random
walk on $(G, \mu)$ with respect to $d_X$. Then for
$P$-a.e. sample path $\omega$ of the
random walk on $(G, \mu)$, for all $\epsilon > 0$,
and all $A$-quasi-geodesic rays $\gamma: R^{+}\to X$ converging
to ${\rm bnd}(\omega)\in \partial X$,
there exists $n_0\in \mathbb{N}$ such that for all $n\geq n_0$, any
$A$-quasi-axis of $\omega_n$ intersects the $\kappa$-neighborhood of 
$\gamma[t_{1}(n),t_{2}(n)]$.
Here $t_{1}(n)$ (resp. $t_{2}(n)$) is the infimum of all real numbers
such that $d_{X}(\gamma(0), \gamma(t_{1}(n)))\geq  \epsilon L_{X}n$ (resp.
$d_{X}(\gamma(0), \gamma(t_{2}(n))) \geq (1 - \epsilon)L_{X}n$).
\end{proposition}

We will apply these results to $G={\rm Mod}(S)$, the mapping class group of
the surface $S$. It acts on the associated complex of curves
$X=\mathcal{C}(S)$, which is Gromov
hyperbolic by work of Masur and Minsky\cite{MM99}, and has the property that any two points in $X\cup \partial X$ can be connected by a geodesic by work of Bowditch \cite{Bo}. 
An element of ${\rm Mod}(S)$ acts as a loxodromic
isometry on $\mathcal{C}(S)$ if and only if it is a pseudo-Anosov.
We will call a subgroup of ${\rm Mod}(S)$ or a measure on ${\rm Mod}(S)$
nonelementary if it is nonelementary for the action on $\mathcal{C}(S)$.

Our main technical result (Proposition \ref{axisrayteich}) is an
analog of Proposition \ref{axesclosetobasepoint} for the action of 
${\rm Mod}(S)$ on the Teichm\"uller space
$\T(S)$ with the Teichm\"uller metric $d_{T}$ and
its \emph{Thurston boundary}
$\mathcal {PML}$ of all measured projective laminations.
We begin with collecting those known results for this action 
we shall use for our goal.

The set $\mathcal{PML}$ contains the invariant subset of
\emph{uniquely ergodic laminations}. 
In the following results of Kaimanovich and Masur
\cite{Kaimanovich-Masur}, 
The notion of a non-elementary probability measure
is the notion discussed above.
\begin{theorem}

\label{Kaimanovich-Masur}
Let $\mu$ be a nonelementary probability measure on
${\rm Mod}(S)$. For $P$-almost every $\omega$,  and every $o\in \T(S)$,
$\omega_{n}o$ converges to a uniquely ergodic point
${\rm bnd}_{T}(\omega)\in \mathcal{PML}$.
\end{theorem}

In other words, there is a $P$-almost everywhere defined
measurable map ${\rm bnd}:G^{\mathbb{N}}\to {\mathcal P\mathcal M\mathcal L}$ 
sending $\omega$ to $\lim_{n\to \infty}\omega_{n}o\in \mathcal{PML}$.
The measure on
$\mathcal{PML}$ defined by
\[\nu={\rm bnd}_{*}P=\lim_{n\to \infty}\mu^{*n}\]
is called the \emph{harmonic measure} for $\mu$. In fact, 
$(\mathcal{PML},\nu)$
is a model for the \emph{Poisson boundary} of $(G,\mu)$
\cite{Kaimanovich-Masur}.
In particular,
$\nu$ is the unique $\mu$ stationary measure on $\mathcal{PML}$:
for every $g\in G$ and $\nu$ measurable $V\subset \mathcal{PML}$ we have
$$\nu(V)=\sum_{g\in G}\mu(g)\nu(g^{-1}V)$$ and consequently for every $n>0$:
$$\nu(V)=\sum_{g\in G}\mu^{*n}(g)\nu(g^{-1}V).$$

The stationarity and the fact that the support of $\mu$ generates $G$
implies that if $\nu(V)>0$ then
$\nu(gV)>0$ for every $g\in G$.
For every $g\in G$ we have:
$$P(\omega:\lim_{n\to \infty} g\omega_{n}\in V)=P(\omega:\lim_{n\to \infty}
\omega_{n}\in g^{-1}V)=\nu(g^{-1}V)$$
If $C_{e,h_{1},...h_{k}}$ is the cylinder subset of $G^\N$ consisting of
$\omega$ with
$\omega_{i}=h_i$ for $i\leq k$ we have
\begin{align}P(C_{e,h_{1},...h_{k}}\cap
{\rm bnd}^{-1}V) &\notag\\
=P(C_{e,h_{1},...h_{k}})P(\omega:\lim_{n\to \infty} h_{k}\omega_{n}\in
V)&=P(C_{e,h_{1},...h_{k}})\nu(h^{-1}_{k}V)
\notag\end{align}
(this is proved in general for the Poisson boundary by Kaimanovich in
Section 3.2 of \cite{Kaimanovich}).
In particular, if $\nu(V)>0$ then for any cylinder subset $C\subset G^\N$
we have $P(C\cap {\rm bnd}^{-1}V)>0$.

The following claim follows from minimality of the action of $G$ on
$\mathcal {PML}$.
\begin{lemma}\label{fullsupport}
The measure $\nu$ has full support on $\mathcal{PML}$.
\end{lemma}
\begin{proof}
Let $V\subset \mathcal{PML}$ be an open set. By minimality,
$\bigcup_{g\in G}gV=\mathcal{PML}$, and hence
$\nu(\bigcup_{g\in G}gV)=1>0$ whence $\nu(gV)>0$ for some $g\in G$. By
stationarity, $\nu(V)>0$.
\end{proof}

The analog of Theorem \ref{Maher-Tiozzo}
for the action of the mapping class group on ${\mathcal T}(S)$ is due to
Tiozzo \cite{Tiozzo}.

\begin{theorem}
\label{Tiozzo-sublinear}
Let $\mu$ be a nonelementary
probability measure on ${\rm Mod}(S)$ with
finite first moment for $d_T$.
Then there exists $L_T>0$ such that for $P$-a.e. sample path $\omega$ one
has
$$\lim_{n\to \infty}\frac{d_{T}(o,\omega_{n}o)}{n}=L_T,$$ and
for any geodesic ray $\tau$ converging to ${\rm bnd}_{T}(\omega)$ we have
$$\lim_{n\to \infty}\frac{d_{T}(\tau(L_Tn),\omega_{n}o)}{n}=0.$$
\end{theorem}

For a pseudo-Anosov element $\phi \in {\rm Mod}(S)$ let $l(\phi)$ be its
translation length in $\T(S)$.
The following result is Theorem 3.1 of \cite{Horbez-Dahmani}.
\begin{theorem}\label{translationlength}
Let $\mu$ be a nonelementary probability measure on ${\rm Mod}(S)$ with
finite support.
Then
$l(\omega_{n})/n\to L_T$ $(n\to \infty)$
for $P$ almost every $\omega$.
\end{theorem}

Reformulating the above discussion in terms of bilateral sample paths
instead we obtain:
\begin{theorem}
\label{bilatreform}
For $\mu$-almost every sample path $\omega \in G^{\mathbb{Z}}$ and every
$x\in \T(S)$ the limits
$${\rm bnd}_{\pm}(\omega)=\lim_{n\to \pm \infty}\omega_{n}x
\in \mathcal{PML}$$ exist,
are independent of $x$, distinct and uniquely ergodic.

There is a geodesic $\tau_{\omega}$ with vertical and horizontal foliations
${\rm bnd}^{\pm}(\omega)$,
and for any unit speed parametrization of $\tau_{\omega}$ we have
$$\lim_{n\to \pm \infty} d_{T}(x,\omega_{n}x)/|n|\to L_T$$ and
$$\lim_{n\to \pm \infty} d_{T}(\tau_{\omega}(L_Tn),\omega_{n}x)/|n|\to
0.$$
The measure $\nu={\rm bnd}_{*}P$ has full support on $\mathcal{PML}$.
Moreover, for any set
$V\subset \mathcal{PML}\times \mathcal{PML}$
with $(\nu \otimes \nu)(V)>0$ and each cylinder subset
$B$ of $G^{\mathbb{Z}}$ we have
$\overline{P}(({\rm bnd}_{-}\times {\rm bnd}_{+})^{-1}(V)\cap B)>0$.
\end{theorem}

From now on, for $o\in \T(S)$ and for almost every
sample path $\omega \in {\rm Mod}(S)^{\mathbb{Z}}$
we denote by $\omega_{\pm}={\rm bnd}_{\pm}(\omega)$ the
limits $\lim_{n\to \pm \infty} \omega_{n}o\in \mathcal {PML}$,
respectively.
Moreover, $\tau_{o,\omega_+}$ denotes the
Teichm\"uller geodesic ray which connects the basepoint
$o$ to $\omega_+$. Recall that this makes sense since for $\overline{P}$ a.e.
sample path $\omega$ the exit point
$\omega_+$ is uniquely ergodic.
Also, from now on, let $L=L_{T}$ be the drift of the $\mu$ with respect to the Teichm\"uller metric $d_{T}$ (see Theorem \ref{Tiozzo-sublinear}). 

Theorem \ref{thm3} can be reformulated as follows.
For a pseudo-Anosov element $\phi$ let $\gamma_{\phi}$ be a unit speed
parametrization of its axis in $\T(S)$ and $l(\phi)$ its translation length
in $\T(S)$. 
Choose moreover a basepoint $o\in {\mathcal T}(S)$.
For $\zeta>0$ and a subset $W$ of ${\mathcal T}(S)$ let
$N_\zeta(W)$ be the $\zeta$-neighborhood of $W$ for the
Teichm\"uller metric.

\begin{theorem}
\label{thickunilrandom}
There is a number $\zeta>0$ with the following property.
Let $W \subset \T(S)$ be an ${\rm Mod}(S)$ invariant open subset containing an axis of a pseudo-Anosov.
Then for each $R>0$,
there exists $c=c(W,R)> 0$ such that
\begin{align}
P\{\omega \in \Mod(S)^{\mathbb{N}}& \mid\omega_n \mbox{ is p-A and }\notag\\
 l(\omega_{n})^{-1}|\{t \in [0,l(\omega_n)]&: \gamma_{\omega_n}(t-R, t+R)
\subset N_{\zeta}W  \}|  > c\} \notag\\
&\to 1\quad (n \to \infty)\notag\end{align}
\end{theorem}

To prove Theorem \ref{thickunilrandom} we first present 
two recurrence results for random geodesics, whose proof we defer until the end of the section.
These results can be viewed as weak versions of Birkhoff's ergodic theorem 
for random rays in Teichm\"uller space.

\begin{proposition}
\label{longthinrare}
There is a number $K>0$ with the following property.
For every $0<a<b$ and $M>0$,
for $P$ almost every sample path $\omega$ there is an $n_0>0$ such that
for all $n>n_0$ $\tau_{o,\omega_{+}}([an,bn])$  contains a connected subsegment of length $M$ contained in $N_{K}{\rm Mod}(S)o.$
\end{proposition}

\begin{proposition}
\label{openrayteich}
Let $W \subset \T(S)$ be an ${\rm Mod}(S)$ invariant open subset that  contains an axis of a pseudo-Anosov.
Then for all $R>0$ there exists
a $\hat c=\hat c(R)>0$ such that for almost every sample path $\omega$ we
have:
\[\lim \inf \frac{1}{T}|\{t\in [0,T]\mid
\tau_{o,\omega_+}[t-R,t+R]\subset W\}|>\hat c\]
where $\omega_+={\rm bnd}_T(\omega)\in \mathcal {PML}.$
\end{proposition}

We will use Proposition \ref{longthinrare} together with 
Proposition \ref{axesclosetobasepoint} to prove our analogue of 
Proposition \ref{axesclosetobasepoint} for ${\mathcal T}(S)$.

\begin{proposition}
\label{axisrayteich}
There exists a number $\zeta>0$ with the following property.
For almost every sample path $\omega$, and each $\epsilon$,
there exists a number $n_{0}>0$
such that for $n\geq n_0$,
$\omega_n$ is a pseudo-Anosov mapping class
with translation length $Ln/2<l(\omega_{n})<2Ln$, and
the axis of $\omega_n$ passes within $\zeta$ of
$\tau_{o,\omega_+}(t)$ for every $t\in [\epsilon Ln,(1-\epsilon) Ln]$.
\end{proposition}
\begin{proof}
The curve complex $X={\mathcal C}(S)$ is a separable,
Gromov hyperbolic geodesic metric space, and any two points in $X\cup \partial X$ can be connected by a geodesic ray \cite{Bo}. The mapping class group
$G={\rm Mod}(S)$ acts on it as
a nonelementary group of isometries.
Pseudo-Anosov elements
act as loxodromics. Thus by
Theorem \ref{Maher-Tiozzo}, there is an $L'>0$ such that
$$d_{{\mathcal C}(S)}(\omega_{n}x,x)/n\to L'$$
for $P$ almost every sample path $\omega$ and every
$x\in {\mathcal C}(S)$.
Moreover, $P$ almost every sample path $\omega$
converges to some $\omega_{+}\in \partial {\mathcal C}(S)$
(the Gromov boundary of ${\mathcal C}(S)$), and there is a unit
speed geodesic ray $\alpha$ starting at $x$ and converging
to $\omega_{+}$ such that
$d(\alpha(L^\prime n),\omega_{n}x)/n\to 0$.
We refer to \cite{Tiozzo,Maher-Tiozzo}
for a detailed discussion with a comprehensive
list of references.

Let $\pi:\T(S)\to {\mathcal C}(S)$ be the coarsely
well-defined map sending $x$ to a shortest curve on $x$.
Then for a uniform constant $A>1$, $\pi$ is $A$
quasi-Lipschitz (i.e.
$d_{{\mathcal C}(S)}(\pi(x),\pi(y))\leq Ad_{\mathcal T}(x,y)+A$ for
all $x,y\in {\mathcal T}(S)$),
and it sends Teichm\"uller geodesics
to $A$ quasi-geodesics in ${\mathcal C}(S)$ \cite{MM99}.
Moreover, if $g\in {\rm Mod}(S)$ is pseudo-Anosov,
and if $\gamma$ is the axis of $g$ in $\T(S)$, then
$\pi(\gamma)$ is an $A$-quasi-axis for the action of
${\rm Mod}(S)$ on ${\mathcal C}(S)$. This
means that it is a $g$-quasi-invariant
$A$-quasi-geodesic in ${\mathcal C}(S)$.

Furthermore, by Theorem 4.2 of \cite{Minsky96}, Teichm\"uller
geodesics in the thick part of Teichm\"uller space are 
uniformly contracting, and this contraction can be traced in the
curve graph \cite{Hamenstaedt10}. This means that 
for each $K>0$ there is a number $D=D(K)>0$, and for every
$\kappa>0$ there is a number 
$\kappa'>0$ with the following property. Let $\alpha_{1},\alpha_{2}$ be
Teichm\"uller geodesics with
$\alpha_{1}(t-D,t+D)\subset N_{K}{\rm Mod}(S)o$ %thick part%
and $d_{{\mathcal C}(S)}(\pi(\alpha_{1}(t)),\pi(\alpha_{2}(t)))<\kappa$;
then $d_{T}(\alpha_{1}(t),\alpha_{2}(t)))<\kappa'$.

Let as before $L$ be the drift of $\mu$ with respect to the Teichm\"uller metric
$d_{T}$.
For $P$ almost every sample path $\omega$
there exists a unit speed geodesic
ray $\tau=\tau_{o,\omega_{+}}$ in $\T(S)$ based at $o$ and
a unit speed geodesic ray $\alpha$ in
${\mathcal C}(S)$ based at $o'=\pi(o)$ such that
\[d_{T}(\tau(Ln),\omega_{n}o)/n\to 0\text{ and }
d_{C(S)}(\alpha(L'n),\omega_{n}o')/n\to 0.\]
Consequently, since $\pi:\T(S)\to \mathcal{C(S)}$ is coarsely Lipschitz we
have
\begin{equation}\label{synchronize}
d_{{\mathcal C}(S)}(\pi(\tau(Ln)),\alpha(L'n))/n\to 0.\end{equation}

Let $\gamma_{n}$ be the axis of $\omega_n$ if $\omega_n$ is
pseudo-Anosov and $o$ otherwise.
Then $\pi(\gamma_{n})$ is an $A$ quasi-axis for the action of
$\omega_n $ on ${\mathcal C}(S)$,
and $\pi\circ \tau$ is an $A$-quasi-geodesic.
By Proposition \ref{axesclosetobasepoint}, there is a number $\kappa>0$ such that
for $P$ a.e. sample path $\omega$ and every $\epsilon>0$ there is an
$n_{0}>0$ such that for each $n\geq n_{0}$, every point $p\in \pi\circ
\tau$ with $$\epsilon L'n \leq d_{{\mathcal C}(S)}(p,o')\leq
(1-\epsilon)L'n$$
is within $\kappa$ of some point of $\pi(\gamma_{n})$.
Now, if $t\in [3\epsilon Ln,(1-3\epsilon) Ln]$
is any integer and $n$ is large enough, using
the estimate (\ref{synchronize}) we have:
$$d_{{\mathcal C}(S)}(o',\pi(\tau(t)))\leq
d_{{\mathcal C}(S)}(o',\alpha(\frac{L'}{L}t))
+d_{{\mathcal C}(S)}(\alpha(\frac{L'}{L}t),\pi(\tau(t)))
$$ $$ \leq \frac{L'}{L}t + \epsilon
L'n\leq (1-3\epsilon) L'n+\epsilon L'n \leq (1-\epsilon) L'n$$
and similarly
$$d_{{\mathcal C}(S)}(o',\pi(\tau(t)))\geq \epsilon L'n$$

Thus, $\pi(\tau(t))$ is within $\kappa$ of some point of $\pi(\gamma_{n})$,
and if $$\tau([t-D,t+D])\in N_{K}{\rm Mod}(S)o$$ we have that $\tau(t)$ is
within $\kappa'$ of some point of $\gamma_{n}$.
Now note that by Proposition \ref{longthinrare}, for $P$ almost every $\omega$
and for all large enough $n$ there are $$t_{1}\in [3\epsilon Ln, 4\epsilon Ln]$$ and $$t_{2}\in [(1-4\epsilon) Ln, (1-3\epsilon) Ln]$$ with 
$$\tau([t_{i}-D,t_{i}+D])\in N_{K}{\rm Mod}(S)o$$
Thus $\tau(t_{i})$ are within $\kappa'$ of some point of $\gamma_n$ and so by \cite{Rafi14}, $\tau([t_{1},t_{2}])$ is within $\kappa''$ of $\gamma_n$ where $\kappa''$ depends only on $\kappa'$.
\end{proof}

Finally, we will use Propositions \ref{axisrayteich} and \ref{openrayteich} to prove 
Theorem \ref{thickunilrandom} and hence Theorem \ref{thm3}. 

\begin{proof}[Proof of Theorem \ref{thickunilrandom}]
Let $\zeta>0$ be the constant guaranteed by
Proposition \ref{axisrayteich}. Denote as before by $L$ the drift of 
the random walk acting on ${\mathcal T}(S)$.

Let $W\subset {\mathcal T}(S)$ be an open ${\rm Mod}(S)$-invariant
set which contains the axis of a pseudo-Anosov element. Let 
$R>0$ and let
$0<\hat c(R+100\zeta)<1/10$ be the constant guaranteed by Proposition
\ref{openrayteich} for $W$ and $R+100\zeta$ in place of $R$.
Let $\epsilon =\hat c/20<1/200$.

For each $N>0$ let $\Omega_{1}(N)$ be the set of all sample paths $\omega$
such that
$$\frac{1}{T}|\{t\in [0,T]:
\tau_{o,\omega_+}[t-R-100\zeta,t+R+100\zeta]\subset W\}|>\hat c$$ for all $T>(1-\epsilon)LN$.

Let $\Omega_{2}(N)$ consist of all sample paths $\omega$ such that for all
$n>N$,
$\omega_{n}$ is a pseudo-Anosov with the following properties.
The translation length of $\omega_n$ is contained in $[Ln/2,2Ln]$,
and its axis passes within $\zeta$
of $\tau_{o,\omega_+}(t)$ for every $t\in [\epsilon Ln,(1-\epsilon) Ln]$.
Define $\Omega(N)=\Omega_{1}(N)\cap \Omega_{2}(N)$.

Let $\omega\in \Omega(N)$ and let $n\geq N$. Since $\omega\in \Omega_2(N)$,
if $t\in [\epsilon Ln,(1-\epsilon) Ln]$ then
the axis of $\omega_n$ passes within $\zeta$ of $\tau_{o,\omega_+}(t)$.
By the definition of $\Omega_1(N)$,
$[\epsilon Ln,(1-\epsilon) Ln]$
has a (not necessarily connected) subset $I_n$ of Lebesgue measure at least
\[\hat c(1-\epsilon)Ln-\epsilon Ln-(2R+200\zeta)>\hat cLn/2\] which is a union of intervals $J$, each of length at least $2R+200\zeta$, such that
$$\tau_{o,\omega_+}(J)\subset W$$
 for each $J$.
 
Let $I'_{n}\subset I_n$ be a maximal subset which consists of intervals each of length at least $2R+100\zeta$ and any two of which are at least $10\zeta$ apart.
Note, $I'_n$ has measure at least $\hat cLn/4$.
For each interval $J$ of $I'_n$, the axis of $\omega_n$ contains a connected subset of length at least $2R+98\zeta$ contained within $\zeta$ of $\tau_{o,\omega_+}(J)\subset W$. Moreover, since any two intervals of $I'_n$ are at least $10\zeta$ apart, the corresponding subsets of  the axis of $\omega_n$
contained in their $\zeta$ neighborhood are disjoint. 
Thus, the axis of $\omega_n$ contains a subset of measure at least  
$\hat cLn/4$ which is a union of intervals of length at least 
$2R+98\zeta$ contained in $N_{\zeta}(W)$.
Consequently, the $R$-interior of this subset, i.e. the set of 
all points which are contained in one of these intervals
and whose distance to the boundary is at least $R$, 
has measure at least 
$98 \zeta\hat c Ln/4(2R+98\zeta)$.

Since $\omega_n$ has translation length at most $2Ln$, 
the proportion of the points $t$ on the axis which are midpoints of 
segments of length $2R$ entirely contained in 
$N_{\zeta}(W)$ is at least $\tilde c=98 \zeta \hat c /(2R+98\zeta)$:

$$l(\omega_{n})^{-1}|\{t\in [0,l(\phi)] :\gamma_{\omega_n}(t-R, t+R)\subset N_{\zeta}(W)\}|> \tilde c.$$

This holds for each $\omega \in \Omega(N)$ and $n\geq N$.
By Proposition \ref{axisrayteich} and Proposition \ref{openrayteich} we
know that $P(\Omega(N))\to 1$ as $N\to \infty$, completing the proof.
\end{proof}

It remains to prove Proposition \ref{openrayteich} and Proposition \ref{longthinrare}. Since by Theorem \ref{Kaimanovich-Masur}, $\omega_+$ is uniquely ergodic for $P$ almost every $\omega$, and by work of Masur \cite{Masur} any two Teichm\"uller geodesics with the same uniquely ergodic vertical foliation are asymptotic, the propositions follow from certain bilateral analogues. 
Namely,
for $p\in \T(S)$ and $\zeta_{1},\zeta_{2}\in 
{\mathcal P\mathcal M\mathcal L}$ defining a Teichm\"uller geodesic with the $\zeta_i$ as vertical and horizontal measured 
geodesic laminations, let $\gamma_{\zeta_{1},\zeta_{2},p}$ be a unit speed parametrization of this geodesic such that $\gamma_{\zeta_{1},\zeta_{2},p}(0)$ is at minimal distance from $p$. We can make this choice 
in a ${\rm Mod}(S)$ equivariant way, i.e. so that $g\gamma_{\zeta_{1},\zeta_{2},p}(t)=\gamma_{g\zeta_{1},g\zeta_{2},gp}(t)$.

For a bilateral sample path $\omega$ converging to distinct uniquely ergodic $\omega^{\pm} \in {\mathcal P\mathcal M\mathcal L}$ and $p\in \T(S)$ write $\gamma_{\omega,p}$ instead of  
$\gamma_{\omega_{-},\omega_{+},p}$. We also write $\gamma_\omega$
for the trace of the axis in ${\mathcal T}(S)$.
Proposition \ref{openrayteich} follows from the following.

\begin{proposition}
\label{thickgeodesic}
Let $W \subset \T(S)$ be an ${\rm Mod}(S)$ invariant open subset that  contains an axis of a pseudo-Anosov.

For every $R>0$ there exists a $\tilde c=\tilde c(R)>0$
such that for  $\overline{P}$ almost every bilateral sample path $\omega
\in \Mod(S)^{\mathbb{Z}}$  we have:
$$\lim \inf_{T\to \infty} \frac{1}{T}|\{t\in [-T,T]:
\gamma_{\omega,o}([t-R,t+R])\subset W\}|>\tilde c$$
\end{proposition}

\begin{proof}
For each $K,R>0$ let $\Omega(K,R)$ be the set of sample paths $\omega$ such that there is a 
Teichm\"uller geodesic with vertical and horizontal laminations $\omega_{+}$ and $\omega_{-}$, and moreover $\gamma_{\omega,o}([-2R,2R])\subset W$ and $d(o,\gamma_\omega)<K$.
\begin{lemma}\label{positivefirst}
There is a $K=K(W)>0$ such that $\overline{P}(\Omega(K,R))>0$ for all $R>0$.
\end{lemma}
\begin{proof}
Let $K(W)$ be large enough so that there exists a pseudo-Anosov element $\phi$ with attractive and repellant (projective) measured foliations $\phi^{+}$ and $\phi^{-}$ with axis $\gamma_{\phi_{-},\phi_{+}}$  passing within $K/2$ of $o$ and contained in $$W\cap N_{K/2}{\rm Mod}(S)o.$$ Let $\gamma_{\phi_{-},\phi_{+},o}$ be the associated parametrization for its axis.
Let $a>0$ be such that the $a$ neighborhood of the axis of $\phi$ is contained in $W\cap N_{K/2}{\rm Mod}(S)o$.

Filling pairs of laminations are an open subset of 
${\mathcal P\mathcal M\mathcal L}\times 
{\mathcal P\mathcal M\mathcal L}$, and any such pair determines a 
Teichm\"uller geodesic with corresponing vertical and horizontal measured foliations.
Moreover, if a sequence of such 
pairs converges to the pseudo-Anosov pair $(\phi_{-},\phi_{+})$, 
then the corresponding geodesics converge locally uniformly 
to $\gamma_{\phi_{-},\phi_{+},o}$. 
Thus for every $R>0$ there are open neighborhoods 
$U^\pm\subset {\mathcal P\mathcal M\mathcal L}$ 
of $\phi^\pm$ such that  for all $\zeta_{1}\in U^+$ and 
$\zeta_{2}\in U^-$ we have $$d(\gamma_{\zeta_{1},\zeta_{2},o}(t),\gamma_{\phi_{-},\phi_{+},o}(t))<a$$ for all $t\in [-2R,2R]$. 

Let $\Omega'(K,R)$ be the set of all sample paths $\omega$ with $\omega^\pm\in U^\pm$. 
By definition, $\Omega'(K,R)\subset \Omega(K,R)$.
Since the $U^\pm$ are open subsets of 
${\mathcal P\mathcal M\mathcal L}$ and the harmonic measure $\nu$ has full support on ${\mathcal P\mathcal M\mathcal L}$,
we have $\nu(U^\pm)>0$ and hence $\overline{P}(\Omega'(K,R))>0$ and thus $\overline{P}(\Omega(K,R))>0$.  
\end{proof}

Let $\sigma: G^{\mathbb{Z}}\to G^{\mathbb{Z}}$ be the left Bernoulli shift:
$\sigma(\omega)_{n}=\omega_{n+1}.$
By basic symbolic dynamics, $\sigma$ is invertible,
measure preserving and ergodic with respect to $\mu^{\Z}$.
Therefore,
\[U=T\circ \sigma \circ T^{-1}\] is invertible, measure preserving and
ergodic with respect to $\overline{P}$.
Note that for each $n\in \mathbb{Z}$,
$$(U\omega)_{n}=\omega^{-1}_{1}\omega_{n+1}$$
and more generally
$$(U^{k}\omega)_{n}=\omega^{-1}_{k}\omega_{n+k}.$$

Let $K>0$ be as in Lemma \ref{positivefirst}.
Since $W$ and $d_T$ are ${\rm Mod}(S)$ invariant, we have that 
$U^{k}\omega \in \Omega(K,R)$ if and only if $\gamma_{\omega,\omega_{k}o}([-2R,2R])\subset W$ and $d(\omega_{k}o,\gamma_\omega)<K$.

Without loss of generality, we can assume that $R>100K$. 
Let $s_{i}(\omega)=d_{T}(\omega_{i}o,\gamma_{\omega,o}(0))$.
By Theorem \ref{Tiozzo-sublinear}, for almost all $\overline{P}$ almost all $\omega$ we have 
$s_{i}(\omega)/i\to L$. 

If $U^{i}\omega \in \Omega(K,R)$
then there is some $t_i(\omega)$ with 
$|t_{i}(\omega)-s_{i}(\omega)|<2K$ and $$\gamma_{o,\omega}[t_{i}(\omega)-2R,t_{i}(\omega)+2R]\subset W.$$

Define $$I(\omega)=\cup_{\{i\mid U^{i}(\omega)\in\Omega(K,R)\}}[t_{i}(\omega)-R,t_{i}(\omega)+R].$$
We need to show that $I(\omega)$ has positive density in $\R$,
ie that $$\lim \inf_{\rho\to \infty}|I(\omega)\cap[-\rho,\rho]|/2\rho>0.$$
Since $$|t_{i}(\omega)-s_{i}(\omega)|<2K$$ whenever $U^{i}\omega \in \Omega(K,R)$,
it suffices to show that 
$$I'(\omega)=\cup_{\{i\mid U^{i}(\omega)\in\Omega(K,R)\}}[s_{i}(\omega)-R,s_{i}(\omega)+R]$$ has positive density in $\R$.

Let $$d=\sup \{d_{T}(o,go)\mid g\in {\rm supp}(\mu)\}$$ (this is finite since $\mu$ has finite support). Then by 
${\rm Mod}(S)$ invariance of $d_{T}$, for all $\omega$ and $k$ we have $d_{T}(\omega_{k+1}o,\omega_{k})<d$.
Thus
for each $$t>d_{T}(o,\gamma_\omega)$$ there is a $n_{t}(\omega)\in \N$ with $$|s_{n_{t}(\omega)}(\omega)-t|<d.$$

If $q>0$ and if $n'$ is a number with $|n_{t}(\omega)-n'|\leq q$ then
$$d_T(\omega_{n_t(\omega)}o,\omega_{n'}o)\leq qd$$ and
hence  ${\rm dist}(t,I'_{n'}(\omega))\leq qd+d$.
Thus if $n^\prime$ is such that
$U^{n^\prime}(\omega)\in \Omega(K,R)$
then $t\in N_{qd+d}(I'(\omega))$.

Now, for $A>0$ assume that $t$ is such that
$d(t,I'(\omega))>Ad+d$. Then there exists an integer
$n>0$ with $U^{i}\omega \notin \Omega(K,R)$ for any $i\in [n-A,n+A]$ and $|t-d(\omega_{n}o,\gamma_{\omega}(0))|<d$.

Let $\Omega_{0}\subset \Mod(S)^{\mathbb{Z}}$ be the $\overline{P}$
full measure set consisting of $\omega$ such that
$d(\omega_{m}o,o)/m\to L$ as $m\to \pm \infty$.
For $\omega\in \Omega_0$ and large enough $t$(depending on $\omega$) we have
$n_{t}(\omega)\in [9t/(10L),11t/(10 L)]$.

Define
\[\Lambda(A,K,R)=\{\omega\mid
U^{i}\omega \notin \Omega(K,R)\text{ for any }|i|\leq A\}.\]

Let $\Upsilon_{A,K,R}(\omega)$ be the set of
$k\in \mathbb{Z}$ with $U^{i}\omega \notin \Omega(K,R)$
for any $$k-A\leq i\leq k+A.$$

By definition, $U^{k}\omega \in \Lambda(A,K,R)$ if and only
if $k\in \Upsilon_{A,K,R}(\omega)$.
Thus any large enough $t$ with $d(t,I'(\omega))>Ad+d$
is contained in $$[d_T(\omega_{k}o,\gamma_{\omega}(0))-d,
d_T(\omega_{k}o,\gamma_{\omega}(0))+d]$$
for some $k<11t/(10L)$ with $k\in \Upsilon_{A,K,R}(\omega)$.
By the Birkhoff ergodic theorem,
$\overline{P}(\Lambda(A,K,R))\to 0$ as $A\to \infty$.

Also by the Birkhoff ergodic theorem,
for almost every $\omega$,
\[\lim_{N\to \infty}|[-N,N]\cap
\Upsilon_{A,K,R}(\omega)|/(2N)= \overline{P}(\Lambda(A,K,R)).\]

Thus for large enough $T$ (depending on $A$)
the Lebesgue measure of the set
\[\{t\in [-T,T]\mid d(t,I'(\omega))>Ad+d\}\] is less than
$3d P(\Lambda(A,K,R))(11T/10L)$.

Therefore the density of
$t\in \mathbb{R}$ with $d(t,I'(\omega))>Ad+d$ is less than
$3dP(\Lambda(A,K,R))/L$ which is less than $1/10$ for large enough $A$.
Thus for large enough $A$ the $Ad+d$ neighborhood of $I'(\omega)$ has density at least $9/10$ in
$\mathbb{R}$ so $I'(\omega)$ itself has density at least
$c=\frac{9}{10(Ad+d)}>0$ in $\mathbb{R}$.
This completes the proof of Proposition \ref{thickgeodesic}.
\end{proof}

Proposition \ref{longthinrare} follows from the following bilateral statement.
\begin{proposition}
\label{longthinrarebilateral}
There is a $K>0$ with the following property.
For every $a<b$ and $M>0$,
for $P$ almost every sample path $\omega$ there is an $n_0>0$ such that
for all $n>n_0$ $\gamma_{\omega,o}(cn,dn)$  contains a connected subsegment of length $M$ contained in $N_{K}{\rm Mod}(S)o.$
\end{proposition}
\begin{proof}
Let $\Omega(M,K,R)$ be the set of sample paths $\omega\in {\rm Mod}(S)^{\Z}$ such that
$d_{T}(o,\gamma_\omega)<R/10$ and
$\gamma_{\omega,o}(t-M,t+M)\subset N_{K}{\rm Mod}(S)o$ for some $t\in (-R/2+M,R/2-M)$.
\begin{lemma}
\label{decaylemma}
There is a $K>0$ such that for all $M>0$ there is an function $f$ with $\lim_{R\to \infty}f(R)=0$ and  
$\overline{P}(\Omega(M,K,R))>1-f(R)$.
\end{lemma}
We first continue with the proof of Proposition \ref{longthinrarebilateral} assuming Lemma \ref{decaylemma} and will prove Lemma \ref{decaylemma} afterwards.

Assume without loss of generality that $a>0$ (notation as in the statement of 
Proposition \ref{longthinrarebilateral}).
Let as before $\Omega_{0}\subset 
{\rm Mod}(S)^\Z$ denote the $\overline{P}$ full measure set of 
all $\omega$ such that $$d_{T}(\omega_{i}o,o)/i\to L$$ and consider $\omega\in \Omega_0$.
Choose $R>0$ large enough so that $$1-f(R)>(b-a)/(10a+10b)$$
Note, $U^{k}\omega \in \Omega(M,K,R)$ if and only if 
$d_{T}(\omega_{i}o,\gamma_\omega)<R/10$ and
$$\gamma_{\omega,\omega_{i}o}(t-M,t+M)\subset N_{K}{\rm Mod}(S)o$$ for some $t\in (-R/2+M,R/2-M)$.
This implies that $$\gamma_{\omega,o}(t_{i}(\omega)-M,t_{i}(\omega)+M)\subset 
N_{K}{\rm Mod}(S)o$$
for some $t_{i}(\omega)$ with $|t_{i}(\omega)-d(\omega_{i}o,\gamma_{\omega,o}(0))|<R/10$.

Let $s_{i}(\omega)=d_{T}(\omega_{i}o,\gamma_{\omega,o}(0))$ and let again
$$d=\sup \{d_{T}(o,go)\mid g\in {\rm supp}(\mu)\}.$$ 
As in the proof of Proposition \ref{thickgeodesic}, note for every $t>d_{T}(o,\gamma_\omega)$ there is some $i(t)$ with 
$|t-s_{i(t)}(\omega)|<d$. Hence, for large enough (depending on $\omega$) $n$ if there is an $i$ with $$U^{i}\omega \in \Omega(M,K,R)$$ and $$(2a+b)n/3\leq s_{i}(\omega)\leq (a+2b)n/3$$ then $$\gamma_{\omega,o}([an,bn])\cap  N_{K}{\rm Mod}(S)o$$ has a connected segment of length $M$.

Moreover, $s_{i}(\omega)/i \to L$. Thus, for large enough $n$, we have $$(2a+b)n/3\leq s_{i}(\omega)\leq (a+2b)n/3$$ for every $i$ with $$\frac{(3a+2b)n}{5L}\leq i\leq\frac{(2a+3b)n}{5L}$$

Hence, if $\gamma_{\omega,o}([an,bn])\cap  N_{K}{\rm Mod}(S)o$ does not have a length $M$ connected segment we have $$U^{i}\omega \notin \Omega(M,K,R)$$ for any 
$$\frac{(3a+2b)n}{5L}\leq i\leq\frac{(2a+3b)n}{5L}$$ If this holds for infinitely many $n$ we have 
$$\lim \inf_{N\to \infty}\frac{|\{i\in [0,N-1]\mid U^{i}\omega \in \Omega(M,K,R)\}|}{N}\leq 1-\frac{b-a}{2a+3b}$$
On the other hand, by the Birkhoff ergodic theorem we have:
$$\lim_{N\to \infty}\frac{|\{i\in [0,N-1]\mid U^{i}\omega \in \Omega(M,K,R)\}|}{N}$$ $$=\overline{P}(\Omega(M,K,R))>1-(b-a)/(10a+10b)$$ giving a contradiction.
\end{proof}
Finally, we prove Lemma \ref{decaylemma}.
\begin{proof}[Proof of  Lemma \ref{decaylemma}]
Clearly the $\overline{P}$ measure of $\omega\in {\rm Mod}(S)^{\Z}$ such that
$d_{T}(o,[\omega_{-},\omega_{+}])<R/10$ converges to 1 with $R$.
Thus it suffices to show that for each $M>0$ the $\overline{P}$ measure of $\omega$ such that $\gamma_{\omega,o}([-R,R])\cap N_{K}{\rm Mod}(S)o $ contains a length $M$ connected subsegment converges to $1$ with $R$.
Let $\Lambda(M,K)$ be the set of $\omega$ such that $d_{T}(o,\gamma_\omega)<K$ and $$\gamma_{\omega,o}([-M,M])\subset N_{K}{\rm Mod}(S)o.$$
The following is similar to Lemma \ref{positivefirst}.

\begin{claim}
There is a $K>0$ such that $\overline{P}(\Lambda(M,K))>0$ for all $M$.
\end{claim} 
\begin{proof}
Let $K>0$ be large enough so that there exists a pseudo-Anosov element $\phi$ with attractive and repellant (projective) measured foliations $\phi^{+}$ and $\phi^{-}$ with axis $\gamma_{\phi_{-},\phi_{+}}$  passing within $K/2$ of $o$ and contained in $N_{K/2}{\rm Mod}(S)o$. Let $\gamma_{\phi_{-},\phi_{+},o}$ be the associated parametrization for its axis.
Then there are open neighborhoods $U^\pm\subset \mathcal{PML}$ of 
$\phi^\pm$ such that  for all $\zeta_{1}\in U^+$ and 
$\zeta_{2}\in U^-$, $d_{T}(\gamma_{\zeta_{1},\zeta_{2},o}(t),\gamma_{\phi_{-},\phi_{+},o}(t))<K/2$ for all $t\in [-2M,2M]$. 

Let $\Lambda'(M,K)$ be the set of all sample paths $\omega$ with $\omega^\pm\in U^\pm$. 
By definition, $\Lambda'(M,K)\subset \Lambda(M,K)$.
Since the $U^\pm$ are open subsets of 
$\mathcal{PML}$ and the harmonic measure $\nu$ has full support on 
$\mathcal{PML}$ we have $\nu(U^\pm)>0$ and hence $\overline{P}(\Lambda'(M,K))>0$ and thus $\overline{P}(\Lambda(M,K))>0$.  
\end{proof}
Note, $U^{i}\omega \in \Lambda(M,K)$ if and only if 
$$d_{T}(o,\gamma_\omega)<K$$ and  $$\gamma_{\omega,\omega_{i}o}([-M,M])\subset N_{K}{\rm Mod}(S)o.$$
Note, $d_{T}(o,\omega_{i}o)\leq di$ and hence if $$U^{i}\omega \in \Lambda(M,K)$$ for some $i$ with
$$0\leq i\leq \frac{R-M-2K}{2d}$$ then $$\gamma_{\omega,o}([-R,R])\cap N_{K}{\rm Mod}(S)o$$ contains a length $M$ connected subsegment. 
By the Birkhoff ergodic theorem, the $\overline{P}$ measure of sample paths $\omega$ such that $U^{i}\omega \notin \Lambda(M,K)$ for all $i$ with 
$$0\leq i\leq \frac{R-M-2K}{2d}$$ converges to $0$ with $R$ completing the proof.
\end{proof}
%%%%%%%%%%%%%%%%%%%%%%%%%%%%%%%%%%%%%%%%%%%%%%%%%%%%%%%%%%%%%%%%%%%%%%%%%%%

%%%%%%%%%%%%%%%%%%%%%%%%%%%%%%%%%%%%%%%%%%%%%%%%%%%%%%%%%%%%%%%%
%%%%%%%%%%%%%%%%%%%%%%%%%%%%%%%%%%%%%%%%%%%%%%%%%%%%%%%%%%%%%

%%%%%%%%%%%%%%%%%%%%%%%%%%%%%%%%%%%%%%%%%%%%%%%%%%%%%%%%%%%%%%%%%%%%%%%%%%%
%%%%%%%%%%%%%%%%%%%%%%%%%%%%%%%%%%%%%%%%%%%%%%%%%%%%%%%%%%%%%%%%%%%%%%%%%%%

\bigskip

\noindent
Hyungryul Baik, Department of Mathematical Sciences, KAIST
291 Daehak-ro, Yuseong-gu, 34141 Republic of Korea\\
Ilya Gekhtman,
Department of Mathematics, Yale University, 10 Hillhouse Ave, New Haven, Connecticut 06520, USA\\
Ursula Hamenst\"adt,
Universit\"at Bonn, Endenicher Allee 60, 53115 Bonn, Germany\\
\noindent
e-mail: \\
Hyungryul Baik: hrbaik@kaist.ac.kr\\
Ilya Gekhtman: ilya.gekhtman@yale.edu\\
Ursula Hamenst\"adt: ursula@math.uni-bonn.de

\end{document}